\documentclass{IEEEojcsys_mod}
\usepackage{amsmath}
\usepackage{amssymb}
\usepackage{cite}
\usepackage{amsfonts}
\usepackage{mathtools}
\usepackage{amsthm}
\usepackage{color}
\usepackage{comment}
\usepackage{xfrac}
\usepackage{xcolor}
\usepackage{enumerate}
\usepackage{url}
\usepackage[normalem]{ulem}
\usepackage{subcaption}
\usepackage{graphicx}
\usepackage{hyperref}

\newtheorem{defn}{Definition}
\newtheorem{assump}{Assumption}
\newtheorem{lem}{Lemma}
\newtheorem{thm}{Theorem}
\newtheorem{prop}{Proposition}
\newtheorem{ex}{Example}
\newtheorem{rem}{Remark}
\begin{document}
\title{On the Stability of Power Systems with AI Data Centers: A Multi-Equilibrium Switched DAE Framework\vspace{0.2cm}}
\author{Vishal Shenoy \affilmark{1}}
\author{Damola Ajeyemi \affilmark{2}}
\author{Mahmoud Abdelgalil \affilmark{1}}
\author{\\Guido Cavraro \affilmark{3}}
\author{Emiliano Dall'Anese \affilmark{2}}
\author{Jorge I. Poveda \affilmark{1}}
\affil{Department of Electrical and Computer Engineering, University of California, San Diego, USA}
\affil{Division of Systems Engineering, Boston University, USA}
\affil{National Laboratory of the Rockies, Golden, CO, USA}
\corresp{CORRESPONDING AUTHOR: Jorge I. Poveda (e-mail: \href{mailto:poveda@ucsd.edu}{poveda@ucsd.edu})}
\authornote{This work was supported in part by the U.S. Department of Energy, Office of Electricity, Advanced Grid Modeling Program, and by Energy, Power, Control and Learning Program of the National Science Foundation}
\begin{abstract}
The rapid growth of  Artificial Intelligence (AI) data centers is introducing dynamic loads into power transmission systems as GPU clusters switch between computational phases with distinct power demands. These dynamics are naturally modeled by switched affine differential-algebraic equations (DAEs), where each mode induces a different equilibrium and consistency space. We study their stability using control and hybrid systems tools. Differential projectors derived from the quasi-Weierstrass form yield an equivalent hybrid representation, from which we characterize the $\omega$-limit set and establish Lyapunov conditions for practical stability with an explicit $\mathcal{KL}$ bound accounting for switching-induced resets. For stochastic inference workloads, we use Foster--Lyapunov techniques to establish probabilistic uniform global recurrence under mode-dependent dwell times and transition probabilities. The framework is validated on the IEEE 39-bus transmission network with co-located AI data centers using high-fidelity electromagnetic transient simulations of deterministic training and stochastic inference workloads.
\end{abstract}
\begin{IEEEkeywords}
Transmission Systems, Switched Systems, Differential Algebraic Equations, data centers, Stochastic Hybrid Dynamical Systems, Foster-Lyapunov Stability, Electro-Magnetic Transient Simulation.
\end{IEEEkeywords}
\maketitle
\section{Introduction}
\label{sec:intro}
 Artificial Intelligence (AI) data centers represent a new class of dynamic loads on the power grid. GPU clusters supporting large-scale training and inference cycle through computational phases with distinct, approximately constant power demands, producing facility-level swings of tens to hundreds of megawatts~\cite{latif2024}. These transitions can also induce frequency and voltage disturbances at the interconnection level~\cite{chen2025datacenter}. With US datacenter demand projected to reach hundreds of gigawatts by 2030~\cite{chen2025datacenter}, understanding the stability implications of these fast, repeated transitions is becoming increasingly important. Existing work has focused primarily on aggregate load modeling, demand response, capacity planning, and voltage regulation~\cite{chen2025datacenter,chen2026voltage,liang2026gputogrid}. In contrast, the stability implications of GPU-induced switching, where each computational mode produces a distinct load level and power-system equilibrium, remain largely unexplored. This gap is particularly relevant given ERCOT's recent open-source dynamic EMT model for AI datacenter loads, which captures converter dynamics, UPS mode transitions, and disturbance ride-through behavior~\cite{ercot2026datacenter}. 

Under the DC power-flow approximation and standard models of synchronous generators, grid-forming inverters, and grid-following resources, a transmission network with co-located AI data centers can be modeled as a switched affine differential-algebraic equation (DAE)~\cite{gross2016solvability,nadeem2022dynamic,nugroho2022observers}. Algebraic constraints enforce bus-level active-power balance, while workload transitions generate mode-dependent affine terms and, consequently, distinct equilibria and consistency spaces. Existing stability theory for switched DAEs largely assumes a common equilibrium or invariant set across modes~\cite{trenn2012switched,LiberzonTrenn2009,tanwani2010observability,nanez2017invariance}. Reachability-based results for switched affine ODEs can accommodate multiple equilibria~\cite{baradaran2020omega}, but do not account for the algebraic constraints of DAEs. Thus, despite the growing use of switched-system methods in power systems~\cite{abdelgalil2025nhls,odunlami2025hybrid,tang2023common}, stability theory for switched affine DAEs with mode-dependent equilibria and consistency spaces remains largely undeveloped.

\vspace{0.1cm}
Motivated by the preceding discussion, this paper develops two complementary stability frameworks for AI datacenter workloads in power transmission systems. AI training, where computational modes follow a deterministic, possibly aperiodic sequence, is modeled as a deterministic hybrid system~\cite{goebel2012hybrid} under average dwell-time switching. AI inference, where mode transitions are driven by unpredictable request arrivals, is modeled as a stochastic hybrid system~\cite{teel2013lyapunov} with mode-dependent dwell times. Stability and recurrence properties of these systems are studied using Lyapunov and Foster--Lyapunov methods, respectively. Specifically, the following are the main contributions of this paper:
\begin{enumerate}
    \item We develop a model of transmission systems with co-located AI data centers as switched affine DAEs with a common system matrix and mode-dependent affine terms, inducing parallel consistency spaces. For sufficiently slow switching, we explicitly characterize the ``practically'' stable $\omega$-limit set of this system.
    \item We develop a hybrid Lyapunov framework establishing (practical) stability under a slow average dwell-time condition on the switching, including an explicit $\mathcal{KL}$ bound that accounts for switching-induced state resets. These results extend switched DAE stability analysis beyond the common-equilibrium setting.
    \item For stochastic inference workloads, we establish uniform global recurrence in probability using Foster--Lyapunov methods, with the recurrent set determined by the inter-equilibrium distances.
    \item We validate the framework on the IEEE 39-bus transmission network using high-fidelity electromagnetic transient simulations, showing that the switched-DAE model captures the fast converter and switching dynamics induced by AI data centers.
\end{enumerate}
The rest of this paper is organized as follows: Section \ref{sec:prelim} presents the preliminaries. Section \ref{sec:power} introduces a DAE model of the power transmission system. Section \ref{sec:formulation} studies stability under training workloads. Section \ref{sec: stochastic} focuses on inference workloads, and finally Section \ref{sec:conclusion} ends with the conclusions.
\section{Preliminaries}
\label{sec:prelim}
\subsection{Notation}
 We use $\mathbb{R}^n$ to denote the $n$-dimensional Euclidean space, $\mathbb{Z}_{\geq 0}$ for the set of nonnegative integers, and $\mathbb{R}_{\geq 0}$ for the set of nonnegative real numbers. We denote the Euclidean norm of $x \in \mathbb{R}^n$ by $\lvert x \rvert$. Given $x \in \mathbb{R}^n$ and a non-empty closed set $\mathcal{A} \subset \mathbb{R}^n$, the distance of $x$ to $\mathcal{A}$ is given by $\lvert x \rvert_{\mathcal{A}} \coloneqq \inf_{y \in \mathcal{A}} \lvert x- y \rvert$. Given $r >0$, the set $\{x \in \mathbb{R}^n \, : \, \lvert x \rvert \leq r \}$ is denoted by $r\mathbb{B}$. For a closed set $S$ and $\epsilon>0$,
$S+\epsilon\mathbb B$ and
$S+\epsilon\mathbb B^\circ$ denote, respectively, $\{x\in\mathbb R^n:|x|_S\leq\epsilon\}$, and $\{x\in\mathbb R^n:|x|_S<\epsilon\}$. Given a matrix $A\in\mathbb{R}^{n\times m}$, we use $|A|$ to denote the induced matrix norm of $A$, defined as $|A|:=\sup_{x\in\mathbb{B}}|Ax|$. Given two vectors $u, v$, their concatenation $[u^\top ~ v^\top]^\top$ is denoted as $(u,v)$. The closure of a set $S \subset \mathbb{R}^n$ is denoted by $\Bar{S}$. The closure of the convex hull of $S$ is denoted as $\overline{\text{co}}~S$. For a symmetric matrix $P$, let $\lambda_{\text{min}}(P)$ denote the minimum eigenvalue. We define $\lambda_{\text{max}}$ analogously. A set-valued mapping $M:\mathbb{R}^p\rightrightarrows\mathbb{R}^n$ is outer semicontinuous (OSC) at $z$ if for each sequence $\{z_i,s_i\}\to(z,s)\in\mathbb{R}^p\times\mathbb{R}^n$ satisfying $s_i\in M(z_i)$ for all $i\in\mathbb{Z}_{\geq0}$, we have $s\in M(z)$. A mapping $M$ is locally bounded (LB) at $z$ if there exists an open neighborhood $N_z\subset\mathbb{R}^p$ of $z$ such that  $M(N_z)$ is bounded. A function $\alpha : \mathbb{R}_{ \geq 0} \to \mathbb{R}_{\geq 0}$ is said to be of class $\mathcal{K}$ if it is zero at zero, continuous, and strictly increasing. It is said to be of class $\mathcal{K}_\infty$ if it is of class $\mathcal{K}$ and unbounded. It is said to be of class $\mathcal{G}_\infty$ if it is continuous, non-decreasing, and unbounded. A function $\beta : \mathbb{R}_{\geq 0} \to \mathbb{R}_{\geq 0}$ is of class $\mathcal{KL}$ if $\beta(\cdot, t)$ is of class $\mathcal{K}$ for each fixed $t \geq 0$ and $t \mapsto \beta(r,t)$ is non-increasing and  $\lim_{t \to \infty}\beta(r,t)=0$ for each fixed $r \geq 0$. Given a closed set $\mathcal{A}$, $\rho : \mathbb{R}^n \to \mathbb{R}_{\geq 0}$ is of class $\mathcal{PD}(\mathcal{A})$ if it is continuous and for each $\delta >0$ and $\Delta >0$ there exists $\Bar{\rho}>0$ such that $\rho(x) \geq \bar{\rho}$ for all $x \in (\mathcal{A} + \Delta \mathbb{B})\setminus (\mathcal{A} + \delta \mathbb{B}^\circ)$. When $\mathcal{A}$ is compact, this condition reduces to the condition $\rho(x) >0$ for all $x \in \mathbb{R}^n \setminus \mathcal{A}$. Given the open set $\mathcal{O} \subset \mathbb{R}^n$, $\rho : \mathbb{R}^n \to \mathbb{R}_{\geq 0}$ is of class $\mathcal{PD}(\mathcal{O})$ if it is continuous and there exists a closed set $\mathcal{A}$ and $\epsilon >0$ such that $\mathcal{A}+ \epsilon \mathbb{B} \subset \mathcal{O}$ and $\rho \in \mathcal{PD}(\mathcal{A})$.
\subsection{Stochastic Hybrid Dynamical Systems}
Following the formalism of \cite{teel2013lyapunov}, a \emph{stochastic hybrid dynamical system} (SHDS) is characterized by the dynamics
\begin{subequations}\label{eqn:shy_sys_def}
\begin{align}
&\xi\in C,~~~~~~\dot{\xi} \in F(\xi), \label{flowsSHDS}\\
&\xi\in D,~~~~ \xi^+ \in G(\xi,v),~~~v\sim\mu\label{jumpsSHDS}
\end{align}
\end{subequations}
where $C,D \subset \mathbb{R}^n$ are the flow and jump sets, respectively, $F : \mathbb{R}^n \rightrightarrows \mathbb{R}^n$ is the flow map, $G : \mathbb{R}^n \times \mathbb{R}^m \rightrightarrows \mathbb{R}^n$ is the jump map, and $\mu$ is a probability distribution associated to a sequence of independent, identically distributed (i.i.d.) random inputs $\{v_i\}_{i \geq 1}$, defined on a probability space $(\Omega, \mathcal{F}, \mathbb{P})$. In this paper, we consider SHDS with stochasticity entering only through the jumps. Namely, whenever $x \in D$ and a jump occurs, the post-jump value of $x$ is drawn from $G(x, v_i)$ with $v_i \sim \mu$.  Since system \eqref{eqn:hy_sys_def} combines continuous-time and discrete-time dynamics, for each sample path its solutions $\xi(t,j)$ are parametrized by a continuous-time index $t\in\mathbb{R}_{\geq0}$, which increases continuously during the flows \eqref{flowsHDS}, and a discrete-time index $j$, which increases by one during the jumps \eqref{jumpsHDS}. Therefore, solutions $\xi(t,j)$ to system \eqref{eqn:hy_sys_def} evolve on stochastic hybrid time domains (SHTDs), denoted as \text{dom} ($\xi$). For the precise definition of a random solution to $\mathcal{H}_s$, including the pathwise feasibility and causal-measurability requirements with respect to the filtration generated by $\{v_i\}$, we refer the readers to \cite[Sec.~2.2]{teel2013lyapunov}.

The probabilistic stability properties of $\mathcal{H}_s$ are studied using the following definition, borrowed from \cite[Sec.~2.3]{teel2013lyapunov}:
\begin{defn} A closed set $\mathcal{A} \subset \mathbb{R}^n$ is said to be \emph{uniformly globally asymptotically stable in probability} (UGAS$_p$) for $\mathcal{H}_s$ if it is both uniformly globally stable in probability and uniformly globally attractive in probability. Uniform global stability in probability requires that for every $\epsilon > 0$ and $\varrho > 0$ there exists $\delta > 0$ such that every random solution $x$ with $|x(0,0)|_{\mathcal{A}} \leq \delta$ satisfies $x(t,j) \in \mathcal{A} + \epsilon \mathbb{B}$ for all $(t,j)$ in its hybrid time domain, with probability at least $1 - \varrho$. Uniform global attractivity in probability requires that for every $\epsilon > 0$, $\varrho > 0$, and $R > 0$ there exists $s \geq 0$ such that every random solution $x$ with $|x(0,0)|_{\mathcal{A}} \leq R$ satisfies $x(t,j) \in \mathcal{A} + \epsilon \mathbb{B}$ for all $(t,j)$ in its domain with $t + j \geq s$, with probability at least $1 - \varrho$.
\hfill \qed
\end{defn}
In some cases, we consider set recurrence (\cite[Sec.~2.4]{teel2013lyapunov}) rather than stability:
\begin{defn} Let $\mathcal{O} \subset \mathbb{R}^n$ be an open set. The set $\mathcal{O}$ is said to be \emph{uniformly globally recurrent} (UGR) for $(C,F,D,G,\mu)$ if for every $\varrho > 0$ and $R > 0$ there exists $s \geq 0$ such that every random solution $x$ with $|x(0,0)|\leq R$ satisfies the following with probability at least $1 - \varrho$: either the hybrid time domain of $x$ contains no $(t,j)$ with $t + j \geq s$, or there exists $(t,j)$ in the domain of $x$ with $t + j \leq s$ such that $x(t,j) \in \mathcal{O}$.
\hfill \qed
\end{defn}
For the case when $G$ does not depend on a random input, system \eqref{eqn:shy_sys_def} becomes a deterministic hybrid system of the form \cite[Ch.~2]{goebel2012hybrid}:
\begin{subequations}\label{eqn:hy_sys_def}
\begin{align}
&\xi\in C,~~~~~~\dot{\xi} \in F(\xi), \label{flowsHDS}\\
&\xi\in D,~~~~ \xi^+ \in G(\xi),\label{jumpsHDS}
\end{align}
\end{subequations}
For these systems, we use the following stability notions:
\begin{defn}
    The set $\mathcal{A}$ is uniformly globally asymptotically stable (UGAS) for system \eqref{eqn:hy_sys_def} if there exists $\beta \in \mathcal{KL}$ such that any solution $\xi$ satisfies $\lvert \xi(t,j) \rvert _{\mathcal{A}} \leq \beta(\lvert \xi(0,0) \rvert_\mathcal{A},t+j)$, for all $(t,j)\in\text{dom}(\xi)$. 
    \hfill \qed
\end{defn}
The following definition is instrumental for the stability analysis of HDS that depend on a small parameter $\varepsilon$.
\begin{defn}\label{defn:SGPAS} 
    The set $\mathcal{A}$ is semi-globally practically asymptotically stable (SGPAS) for an $\epsilon$-parameterized HDS $\mathcal{H}_{\epsilon}$ as $\epsilon \to 0$ if there exists $\beta \in \mathcal{KL}$ and for each $\mu,\nu >0$ there exists a $\epsilon^* >0$ such that, for each $\epsilon \in (0,\epsilon^*)$, any solution $\xi$ with $\lvert \xi(0,0) \rvert_{\mathcal{A}} \leq \mu$ also satisfies  $\lvert \xi(t,j) \rvert_{\mathcal{A}} \leq \beta(\lvert \xi(0,0) \rvert_{\mathcal{A}},t+j) + \nu$, for all $(t,j) \in \text{dom}(\xi)$. \hfill \qed
\end{defn}

\section{Power Transmission System with AI Datacenter Loads}
\label{sec:power}
\vspace{1mm}
We consider a power transmission system characterized by a set of buses $\mathcal{N}=\{1,\ldots,N\}$ and a set of lines $\mathcal{E} \subset \mathcal{N} \times \mathcal{N}$. The buses are partitioned into synchronous generator (SG) buses
$\mathcal{G}$, grid forming (GFM) inverter-based resources (IBR) buses $\mathcal{I}$, grid following (GFL) IBR buses $\mathcal{F}$, and
load buses $\mathcal{L}$. 
The subset $\mathcal{L}_{dc}\subseteq\mathcal{L}$ hosts a collection of AI training data centers. We use $\mathcal L_0:=\mathcal L\setminus\mathcal L_{dc}$ for the set of conventional load buses without data centers.
Each bus $n\in\mathcal{L}_{dc}$ is paired with a campus GFM bus
$n'\in\mathcal{I}_{dc}\subseteq\mathcal{I}$. For each bus $ i \in \mathcal{N}$, we denote its set of neighboring buses by $\mathcal{N}_i = \{j \in \mathcal{N} \, : \, (i,j) \in \mathcal{E}\}$. {Let $\delta_{\mathrm{bus}}=[\delta_1,\ldots,\delta_N]^\top\in\mathbb{R}^N$ denote the vector of voltage phase angles in a synchronous reference frame. We choose bus $N$ as the angular reference and use the relative-angle coordinates
$\theta=[\delta_1-\delta_N,\ldots,\delta_{N-1}-\delta_N]^\top\in\mathbb{R}^{N-1}$. At each dynamic bus $n\in\mathcal G\cup\mathcal I$, $\omega_n:=\dot\delta_n$ denotes the angular-frequency deviation from nominal frequency. The variables $\omega_n$ at the algebraic buses are the quasi-steady frequency estimates defined by the frequency divider below; collect all of these variables in $\omega=[\omega_1,\ldots,\omega_N]^\top$.}
\vspace{0.2cm}
\par \noindent
\subsection{Power Flow and DC Approximation:}  {We assume fixed voltage magnitudes, negligible line resistances, and sufficiently small voltage-angle differences, so that the network is represented by a DC power-flow model.} Under this approximation, the active-power injection at bus $i$ is given by  {$p_i(\theta) = \sum_{j \in \mathcal{N}_i} b_{ij} (\theta_i - \theta_j)$}, where $b_{ij} = 1/x_{ij} >0$ denotes the line susceptance of the line associated with the edge $(i,j) \in \mathcal{E}$,  {with the convention $\theta_N=0$}. Stacking the vector of active power injections  {$p(\theta) = [p_1(\theta), \cdots, p_N(\theta)]^\top$}, the DC power-flow approximation can be compactly written as 
$p(\theta)=\bar{B}\theta$ with $\bar{B}=B\mathcal{I}_0$, where
$B\in\mathbb{R}^{N\times N}$ is the susceptance Laplacian and
$\mathcal{I}_0=[I_{N-1}~\mathbf{0}]^\top\in\mathbb{R}^{N\times(N-1)}$ embeds
the reduced angle vector into the full $N$-bus angle vector with the
reference-bus angle set to zero. For compactness, we define selector matrices $S_{\mathcal{G}} \in 
\{0,1\}^{N \times \lvert \mathcal{G} \rvert },S_{\mathcal{I}} \in \{0,1\}^{N \times \lvert \mathcal{I} \rvert } ,S_{\mathcal{F}} \in \{0,1\}^{N \times \lvert \mathcal{F} \rvert }, \text{ and } S_{\mathcal{L}}\in  \{0,1\}^{N \times \lvert \mathcal{L} \rvert }$ to 
extract bus-specific variables from full $N$-bus vectors.  {We  use $S_{\mathcal L_0}$ and $S_{\mathcal L_{dc}}$ for the two subsets of $\mathcal L$, and, without loss of generality, order the load-bus selector as $S_{\mathcal L}=[S_{\mathcal L_0}~S_{\mathcal L_{dc}}]$.} Therefore, we set $\omega_{\mathcal{G}} \coloneqq S_{\mathcal{G}}^\top \omega $ and define the vectors $\omega_\mathcal{I}, \omega_\mathcal{F}$, and $\omega_\mathcal{L}$ analogously.  {Consistently with the ordering used below, define
$S_{\mathcal{IG}}:=[S_{\mathcal I}~S_{\mathcal G}]$ and
$\omega_{\mathcal{IG}}:=[\omega_{\mathcal I}^\top~\omega_{\mathcal G}^\top]^\top=S_{\mathcal{IG}}^\top\omega$.}  
\par \noindent
 {Let $m:=|\mathcal G|+|\mathcal I|$. Since the reference bus belongs to $\mathcal G\cup\mathcal I$, define
$\bar S_{\mathcal{GI}}\in\{0,1\}^{(N-1)\times(m-1)}$ to select the reduced-angle coordinates associated with $(\mathcal G\cup\mathcal I)\setminus\{N\}$, and define
$\bar S_{\mathcal F}:=\mathcal I_0^\top S_{\mathcal F}\in\{0,1\}^{(N-1)\times|\mathcal F|}$ and
$\bar S_{\mathcal L}:=\mathcal I_0^\top S_{\mathcal L}\in\{0,1\}^{(N-1)\times|\mathcal L|}$. Thus,
$\theta_{\mathcal{GI}}:=\bar S_{\mathcal{GI}}^\top\theta$,
$\theta_{\mathcal F}:=\bar S_{\mathcal F}^\top\theta$, and
$\theta_{\mathcal L}:=\bar S_{\mathcal L}^\top\theta$. Let
$S_{\mathcal{GI}}^0:=\mathcal I_0\bar S_{\mathcal{GI}}$ select the nonreference $\mathcal G\cup\mathcal I$ buses from full bus vectors, and let $e_N\in\mathbb R^N$ denote the $N$th canonical vector. The matrices $\mathcal R_{\mathcal I}:=(S_{\mathcal{GI}}^0)^\top S_{\mathcal I}
-\mathbf 1_{m-1}e_N^\top S_{\mathcal I}$ and 
$\mathcal R_{\mathcal G}:=(S_{\mathcal{GI}}^0)^\top S_{\mathcal G}
-\mathbf 1_{m-1}e_N^\top S_{\mathcal G}$
give the relative-angle kinematics
$\dot\theta_{\mathcal{GI}}=\mathcal R_{\mathcal I}\omega_{\mathcal I}
+\mathcal R_{\mathcal G}\omega_{\mathcal G}$. Finally, set
$S_{\mathcal{LF}}:=[S_{\mathcal F}~S_{\mathcal L}]$.}
The following assumption is standard:
\begin{assump}\label{asm:network}
 {The network $(\mathcal{N}, \mathcal{E})$ is undirected and connected. Moreover,} $\mathcal{G}\cup\mathcal{I}\neq\emptyset$, and bus
$N\in\mathcal{G}\cup\mathcal{I}$.
\hfill \qed
\end{assump}
\subsection{Models of the Devices}
The dynamics of each synchronous generator $n\in\mathcal{G}$ are approximated by the following turbine-governor model \cite{ajala2020library}:
\begin{align*}
     {\dot{\delta}_n}& {=\omega_n}, \\
M_n\dot{\omega}_n&=-D_n\omega_n+p_{m,n}-p_{\ell,n}-p_n(\theta), \\
\tau_n\dot{p}_{m,n}&=-p_{m,n}-R_n^{-1}\omega_n+p_n^{\mathrm{ref}},   
\end{align*}
where$M_n >0$ and $D_n >0$ denote the inertia and damping coefficients, respectively. The variable $p_{m,n}$ represents the turbine mechanical power, and  {$p_{\ell,n}$} the local load. Finally, $\tau_n>0$ denotes the turbine time constant,  {$R_n >0$ denotes the frequency--power droop coefficient, $R_n^{-1}$ denotes its inverse,} and $p_n^{\text{ref}}$ denotes the active power reference. 
The dynamics of the GFM buses $n\in\mathcal{I}$ are modeled by the following equations:
\begin{align*}
     {\dot{\delta}_n} {=\omega_n},~~M_n\dot{\omega}_n= -D_n\omega_n+p_n^{\mathrm{ref}}-p_{\ell,n}-p_n(\theta),    
\end{align*}
with virtual inertia and damping parameters $M_n >0$ and $D_n>0$, respectively \cite{dorfler2023}.
On the other hand, since the phase-locked loops (PLL) dynamics of the GFL  {IBRs} are significantly faster than electromechanical  {timescales}, the behavior of the
GFL buses $n\in\mathcal{F}$ can be characterized by the following algebraic equations: 
\begin{equation*}
    0=p_n^{\mathrm{ref}}-K_n\omega_n-p_{\ell,n}-p_n(\theta),    
\end{equation*}
 {where $K_n\geq0$ is the GFL frequency-response coefficient. Similarly, the conventional load buses $n \in\mathcal{L}_0$ are modeled by the algebraic equation $0=-p_{\ell,n}-p_n(\theta)$.}
\subsection{Models of the Datacenter Loads:} To model the loads of the data centers, we let each $n\in\mathcal{L}_{dc}$ host a cluster of  {graphics} processing units (GPUs) cycling through one of two mode sets depending on whether training or inference workloads are being considered. For \emph{training}, we consider the following set of modes:
\begin{equation}
     {\mathcal{Q}_{\mathrm{tr}}:=\{\texttt{fw},\,\texttt{bw},\,\texttt{ar},\,\texttt{id}\}}
    \label{eq:modes_ps}
\end{equation}
corresponding to \emph{forward pass}, \emph{backward pass}, \emph{activation recompute}, and \emph{idle} modes of the training phase \cite{patel2024,latif2024}:
\begin{enumerate}
    \item \textit{Forward pass} (\texttt{fw}): the input batch is
    propagated layer-by-layer through the network to compute
    activations, drawing power $p_{dc}^{\texttt{fw}}$.
    \item \textit{Backward pass} (\texttt{bw}): gradients are
    propagated layer-by-layer back through the network via backpropagation, drawing power $p_{dc}^{\texttt{bw}}$.
    \item \textit{Activation recompute} (\texttt{ar}): activations
    discarded during the forward pass are recomputed on demand to
    reduce memory footprint, drawing power $p_{dc}^{\texttt{ar}}$ \cite[Section 4.3]{go2025characterizing}.
    \item \textit{Idle} (\texttt{id}): the cluster awaits the next
    batch or optimizer step, drawing baseline power
    $p_{dc}^{\texttt{id}}$.
\end{enumerate}
 {For the operating regime considered in this paper, we assume that the training-mode power draws satisfy
$p_{dc}^{\texttt{fw}} > p_{dc}^{\texttt{bw}} > p_{dc}^{\texttt{ar}} > p_{dc}^{\texttt{id}}>0$.}

\vspace{0.2cm}
On the other hand, during \emph{inference}, each user request drives the cluster of GPUs through three distinct phases:
\begin{enumerate}
    \item \textit{Prefill} (\texttt{pre}): all prompt tokens are processed
    in a single forward pass, drawing peak  {power $p_{\mathrm{dc}}^{\texttt{pre}}$}~\cite[Sec.~2.1]{distserve2024,
    patel2024}.
    \item \textit{Decode} (\texttt{dec}): tokens are generated
     {autoregressively, drawing the moderate power level $p_{\mathrm{dc}}^{\texttt{dec}}$}~\cite[Sec.~2.1]{distserve2024,
    patel2024}.
    \item \textit{Idle} (\texttt{id}): the cluster awaits the next request, drawing baseline power
    $p_{\mathrm{dc}}^{\texttt{id}}$.
\end{enumerate}
This leads to the mode set 
\begin{equation}
    \mathcal{Q}_{\text{in}}:=\{\texttt{pre}, \texttt{dec}, \texttt{id}\}
    \label{eq:modes_inference}
\end{equation}
 {For the operating regime considered in this paper, we assume that $p_{\mathrm{dc}}^{\texttt{pre}} > p_{\mathrm{dc}}^{\texttt{dec}}
> p_{\mathrm{dc}}^{\texttt{id}} > 0$.} As in the training phase, each per-mode power level is treated as
constant  {under the modeling assumption that within-mode fluctuations at millisecond time scales have a negligible effect on the electromechanical response of interest}. The empirical studies performed in ~\cite[Section 4.2]{niu2026tokenpowerbench} and
\cite[Figure 3]{maliakel2026characterizing} report stable per-mode power signatures across model families and insensitivity of decode power to GPU
frequency across five model sizes,  {supporting the use of a constant $p_{\mathrm{dc}}^{\texttt{dec}}$ at the}
electromechanical timescale. 
\vspace{0.1cm}
\par \noindent
For both training and inference workloads, the active power drawn per mode is well approximated by a constant (see \cite{patel2024, latif2024} for training and \cite{niu2026tokenpowerbench, maliakel2026characterizing} for inference). Hence, the active-power balance at the bus $n\in\mathcal{L}_{dc}$ under the mode $q$ is given by:
\begin{equation}\label{eq:dcbal}
    0=-p_{dc,n}^q - K_{dc,n}\,\omega_n^{dc}-p_n(\theta),
\end{equation}
where  {$p_{dc,n}^q>0$ denotes the total demand at bus $n$ in mode $q$ (including any constant co-located demand), $K_{dc,n}\geq 0$ is a constant frequency-response coefficient,} and $\omega_n^{dc}$ is the local frequency, defined below.  {When $|\mathcal L_{dc}|>1$, the single variable $q$ represents a common aggregate mode and therefore imposes synchronous mode switching across the modeled data centers. Independent switching can instead be represented, for workload type $w\in\{\mathrm{tr},\mathrm{in}\}$, by taking the hybrid mode in the product set $\mathcal Q_w^{|\mathcal L_{dc}|}$.}
\subsection{Frequency Divider:} To relate the frequencies of GFL IBR and datacenter buses to those of the remaining buses, we use the frequency-divider approach~\cite{milano2016frequency}. Since GFL dynamics are much faster than electromechanical timescales, their frequencies are approximated as algebraic functions of the swing-bus frequencies.  {Define the row selectors
\[
C_{\mathcal F}:=\begin{bmatrix}I_{|\mathcal F|}&0_{|\mathcal F|\times|\mathcal L|}\end{bmatrix},
\qquad
C_{dc}:=\begin{bmatrix}0_{|\mathcal L_{dc}|\times|\mathcal F|}&S_{\mathcal L_{dc}}^\top S_{\mathcal L}\end{bmatrix}.
\]
Under Assumption~\ref{asm:network}, the grounded Laplacian
$S_{\mathcal{LF}}^\top B S_{\mathcal{LF}}$ is positive definite. Hence, the full frequency-divider matrix}
\begin{equation}\label{eq:fdiv_ps}
     {\widehat H:=-\bigl(S_{\mathcal{LF}}^\top B S_{\mathcal{LF}}\bigr)^{-1}
       S_{\mathcal{LF}}^\top B S_{\mathcal{IG}}
       =\begin{bmatrix}\widehat H_{\mathcal I}&\widehat H_{\mathcal G}\end{bmatrix}}
\end{equation}
 {is well defined \cite[Sec.~7A]{ajeyemi2026reachability} and maps the $\mathcal I\cup\mathcal G$ frequencies to all $\mathcal F\cup\mathcal L$ bus frequencies. In particular, define
$H:=C_{\mathcal F}\widehat H
=\begin{bmatrix}H_{\mathcal I}H_{\mathcal G}\end{bmatrix}$ and
$H^{dc}:=C_{dc}\widehat H
=\begin{bmatrix}H_{\mathcal I}^{dc} H_{\mathcal G}^{dc}\end{bmatrix}$.
Then,
$\omega_{\mathcal F}=H_{\mathcal I}\omega_{\mathcal I}
+H_{\mathcal G}\omega_{\mathcal G}$ and
$\omega_{dc}:=[\omega_n^{dc}]_{n\in\mathcal L_{dc}}=H_{\mathcal I}^{dc}\omega_{\mathcal I}
+H_{\mathcal G}^{dc}\omega_{\mathcal G}$.
We retain $\omega_{\mathcal F}$ and $\omega_{dc}$ as explicit algebraic states. }
\vspace{2mm}
\subsection{Switched Affine DAE Model:} Let $\xi_d=(\theta_{\mathcal{GI}}^\top,\omega_{\mathcal{I}}^\top,
\omega_{\mathcal{G}}^\top,p_{m,\mathcal{G}}^\top)^\top$ and
$\xi_a=(\theta_{\mathcal{F}}^\top,\theta_{\mathcal{L}}^\top,
\omega_{\mathcal{F}}^\top,\omega_{dc}^\top)^\top$, and  {$\xi=[\xi_d^\top~\xi_a^\top]^\top$. Its differential and algebraic dimensions are
$n_d=(m-1)+|\mathcal I|+2|\mathcal G|$ and
$n_a=2|\mathcal F|+|\mathcal L|+|\mathcal L_{dc}|$, respectively.} Using \eqref{eq:fdiv_ps}, the overall dynamics of the system are given by the following DAE: 
\begin{equation}\label{eq:sadae_ps}
    E\dot{\xi}=A\xi + b_q ,\quad
    E=\mathrm{blkdiag}(I_{n_d},\, {\mathbf{0}_{n_a\times n_a}}),
\end{equation}
with block matrix
$A=\left[\begin{smallmatrix}A_{11}&A_{12}\\A_{21}&A_{22}
\end{smallmatrix}\right]$ given by

\vspace{-0.3cm}
\begin{subequations}
\begin{align}
A_{11} &\!=\!
 {\begin{bmatrix}
\mathbf{0} & \mathcal{R}_{\mathcal{I}} & \mathcal{R}_{\mathcal{G}} & \mathbf{0} \\
{-}M_{\mathcal{I}}^{-1}S_{\mathcal{I}}^\top\bar{B}\bar{S}_{\mathcal{GI}} &
{-}M_{\mathcal{I}}^{-1}D_{\mathcal{I}} & \mathbf{0} & \mathbf{0} \\
{-}M_{\mathcal{G}}^{-1}S_{\mathcal{G}}^\top\bar{B}\bar{S}_{\mathcal{GI}} &
\mathbf{0} & {-}M_{\mathcal{G}}^{-1}D_{\mathcal{G}} & M_{\mathcal{G}}^{-1} \\
\mathbf{0} & \mathbf{0} & {-}T_{\mathcal{G}}^{-1}R_{\mathcal{G}}^{-1} &
{-}T_{\mathcal{G}}^{-1}
\end{bmatrix}}\!,\label{eq:A11_ps} \\
A_{12} &\!=\!
 {\begin{bmatrix}
\mathbf{0} & \mathbf{0} & \mathbf{0} & \mathbf{0}\\
{-}M_{\mathcal{I}}^{-1}S_{\mathcal{I}}^\top\bar{B}\bar{S}_{\mathcal{F}} &
{-}M_{\mathcal{I}}^{-1}S_{\mathcal{I}}^\top\bar{B}\bar{S}_{\mathcal{L}} &
\mathbf{0} & \mathbf{0}\\
{-}M_{\mathcal{G}}^{-1}S_{\mathcal{G}}^\top\bar{B}\bar{S}_{\mathcal{F}} &
{-}M_{\mathcal{G}}^{-1}S_{\mathcal{G}}^\top\bar{B}\bar{S}_{\mathcal{L}} &
\mathbf{0} & \mathbf{0}\\
\mathbf{0} & \mathbf{0} & \mathbf{0} & \mathbf{0}
\end{bmatrix}}\!,\label{eq:A12_ps} 
\end{align}
\begin{align}
A_{21} &\!=\!
 {\begin{bmatrix}
{-}S_{\mathcal{F}}^\top\bar{B}\bar{S}_{\mathcal{GI}} &
\mathbf{0} & \mathbf{0} & \mathbf{0}\\
{-}S_{\mathcal L_0}^\top\bar{B}\bar{S}_{\mathcal{GI}} &
\mathbf{0} & \mathbf{0} & \mathbf{0}\\
{-}S_{\mathcal{L}_{dc}}^\top\bar{B}\bar{S}_{\mathcal{GI}} &
\mathbf{0} & \mathbf{0} & \mathbf{0}\\
\mathbf{0} & {-}H_{\mathcal{I}} & {-}H_{\mathcal{G}} & \mathbf{0}\\
\mathbf{0} & {-}H_{\mathcal{I}}^{dc} & {-}H_{\mathcal{G}}^{dc} & \mathbf{0}
\end{bmatrix}}\!,\label{eq:A21_ps}
\end{align}
\begin{align}
A_{22} &\!=\!
 {\begin{bmatrix}
{-}S_{\mathcal{F}}^\top\bar{B}\bar{S}_{\mathcal{F}} &
{-}S_{\mathcal{F}}^\top\bar{B}\bar{S}_{\mathcal{L}} &
{-}K_{\mathcal{F}} & \mathbf{0}\\
{-}S_{\mathcal L_0}^\top\bar{B}\bar{S}_{\mathcal{F}} &
{-}S_{\mathcal L_0}^\top\bar{B}\bar{S}_{\mathcal{L}} &
\mathbf{0} & \mathbf{0}\\
{-}S_{\mathcal{L}_{dc}}^\top\bar{B}\bar{S}_{\mathcal{F}} &
{-}S_{\mathcal{L}_{dc}}^\top\bar{B}\bar{S}_{\mathcal{L}} &
\mathbf{0} & {-}K_{dc}\\
\mathbf{0} & \mathbf{0} & I_{|\mathcal{F}|} & \mathbf{0}\\
\mathbf{0} & \mathbf{0} & \mathbf{0} & I_{|\mathcal{L}_{dc}|}
\end{bmatrix}}\!,\label{eq:A22_ps}
\end{align}
and  {$b_q=[b_d^\top~(b_a^q)^\top]^\top$}, where
\begin{align}
b_d &=
 {\begin{bmatrix}
\mathbf{0}_{m-1} \\
M_{\mathcal{I}}^{-1}\bigl(p_{\mathcal{I}}^{\mathrm{ref}} - p_{\ell,\mathcal{I}}\bigr) \\
{-}M_{\mathcal{G}}^{-1}p_{\ell,\mathcal{G}} \\
T_{\mathcal{G}}^{-1}p_{\mathcal{G}}^{\mathrm{ref}}
\end{bmatrix}},~~~
b_a^q =
 {\begin{bmatrix}
p_{\mathcal{F}}^{\mathrm{ref}} - p_{\ell,\mathcal{F}} \\
{-}p_{\ell,\mathcal L_0} \\
{-}p_{dc}^q \\
\mathbf{0}_{|\mathcal F|} \\
\mathbf{0}_{|\mathcal L_{dc}|}
\end{bmatrix}}.
\end{align}
\end{subequations}
Here $M_{\mathcal{I}}=\mathrm{diag}(M_i)_{i\in\mathcal{I}}$ and
$D_{\mathcal{I}}=\mathrm{diag}(D_i)_{i\in\mathcal{I}}$ are constant diagonal
matrices.  {The remaining  matrices are $M_{\mathcal G}=\mathrm{diag}(M_n)_{n\in\mathcal G}$,
$D_{\mathcal G}=\mathrm{diag}(D_n)_{n\in\mathcal G}$,
$T_{\mathcal G}=\mathrm{diag}(\tau_n)_{n\in\mathcal G}$,
$R_{\mathcal G}=\mathrm{diag}(R_n)_{n\in\mathcal G}$,
$K_{\mathcal F}=\mathrm{diag}(K_n)_{n\in\mathcal F}$,
$K_{dc}=\mathrm{diag}(K_{dc,n})_{n\in\mathcal L_{dc}}$. The last two are the frequency-sensitivity matrices. The vector $p_{dc}^q$ stacks $p_{dc,n}^q$ over $n\in\mathcal L_{dc}$. For each applicable bus set $\mathcal X$, $p_{\ell,\mathcal X}$ and $p_{\mathcal X}^{\mathrm{ref}}$ denote the corresponding stacked local loads and power references.}
The differential affine term $b_d$ is constant, collecting power references
scaled by the constant matrices $M_{\mathcal{I}}^{-1}$, $M_{\mathcal{G}}^{-1}$,
and $T_{\mathcal{G}}^{-1}$.
On the other hand, the algebraic affine term $b_a^q$ carries $-p_{dc}^q$ in the datacenter
power-balance rows, with remaining entries determined by constant references. 
 {Assumption~\ref{asm:network} also guarantees that $A_{22}$ is nonsingular as shown in~\cite{ajeyemi2026reachability}. Furthermore, $A$ is nonsingular; to see this, the first block row of $A\xi=0$ implies that all $\mathcal I\cup\mathcal G$ frequencies have a common value $c$. Since $\widehat H\mathbf 1=\mathbf 1$ (which follows from $B\mathbf 1=0$), the divider equations give $\omega_{\mathcal F}=c\mathbf 1$ and $\omega_{dc}=c\mathbf 1$. The governor equations give $p_{m,\mathcal G}=-R_{\mathcal G}^{-1}c\mathbf 1$. Summing the homogeneous power-balance equations over all buses and using $\mathbf 1^\top B=0$ yields $\left(\mathbf 1^\top D_{\mathcal I}\mathbf 1
+\mathbf 1^\top(D_{\mathcal G}+R_{\mathcal G}^{-1})\mathbf 1
+\mathbf 1^\top K_{\mathcal F}\mathbf 1
+\mathbf 1^\top K_{dc}\mathbf 1\right)c=0$. Thus $c=0$; the bus injections, relative angles, and mechanical-power states are subsequently zero. Hence the power-system DAE is regular, has differentiation index one, and has a unique equilibrium in every mode.}
\vspace{2mm}
\par\noindent
\subsection{Switching Signals:}
In \eqref{eq:sadae_ps}, the mode $q\in\mathcal{Q}:=\mathcal{Q}_{\mathrm{tr}}\cup\mathcal{Q}_{\mathrm{in}}$ represents training or inference workloads. In both cases, the active mode is determined by a right-continuous, piecewise-constant switching signal $t\mapsto\sigma(t)$, subject to different conditions for training and inference:
\begin{enumerate}[(a)]
\item \emph{Training Workloads:} The class of switching signals  {$\sigma_{\mathrm{tr}}:\mathbb{R}_{\geq0}\to\mathcal{Q}_{\mathrm{tr}}$} used in training workloads is denoted by $\mathcal{S}_{\text{tr}}$. Since switching patterns in AI training workloads are approximately deterministic, but
not necessarily periodic \cite{latif2024}, we assume only a mild average dwell-time bound on the number of mode transitions $N_{\sigma}(t_1,t_2)$ that can occur over any finite time
interval $[t_1,t_2]$. In particular, a signal  {$\sigma_{\mathrm{tr}}\in\mathcal{S}_{\text{tr}}$} is said to satisfy the \emph{average dwell-time} (ADT) bound if for any two times $t_2 \geq t_1 \geq 0$ in the domain of  {$\sigma_{\mathrm{tr}}$} we have 
\begin{equation}\label{ADTeq}
 N_{\sigma}(t_1,t_2) \leq N_0 + \frac{t_2-t_1}{\tau_d},~~N_0 \geq 1,~\tau_d>0.  
\end{equation}
Signals in $\mathcal{S}_{\mathrm{tr}}$ may be aperiodic, allow arbitrary mode ordering, and permit finitely many consecutive switches, capturing the variable durations and occasional switching bursts observed in training workloads; see, e.g., \cite[Figures 1, 5, 6]{golden2025prism}.
\item \emph{Inference Workloads}: The class of switching signals $\sigma_{\mathrm{in}}:\mathbb{R}_{\geq 0}\to\mathcal{Q}_{\mathrm{in}}$ modeling inference workloads is denoted by $\mathcal{S}_{\mathrm{in}}$. In contrast to training workloads, inference workloads are driven by user requests, whose timing is inherently stochastic \cite{yang2024queueing,burstgpt2024}. Motivated by this behavior, we model signals in $\mathcal{S}_{\mathrm{in}}$ as stochastic switching signals, with each mode $q\in\mathcal{Q}_{\mathrm{in}}$ assigned a probability distribution governing its dwell time. The stochastic assumptions underlying this model are detailed in Section~\ref{sec: stochastic}.
\end{enumerate}

In the next section, we focus on the stability analysis of system \eqref{eq:sadae_ps} for training workloads subject to switching signals belonging to $\mathcal{S}_{\text{tr}}$. The analysis for inference workloads with switching signals in  {$\mathcal{S}_{\mathrm{in}}$} is deferred to Section \ref{sec: stochastic}.
\section{Stability Analysis under Training Workloads: Deterministic Hybrid Systems}
\label{sec:formulation}
\vspace{0.1cm}
To study the regularity and stability properties of the power transmission dynamics \eqref{eq:sadae_ps} under training workloads, we first reformulate the system as a  deterministic hybrid dynamical system of the form \eqref{eqn:hy_sys_def} with $\mathcal{Q}_{\text{in}}=\emptyset$. In particular, for the purpose of generality, we consider a generalized model of \eqref{eq:sadae_ps}, given by the following switched affine DAE:
\begin{equation}
    \label{eqn: affine_DAE}
    E_q \Dot{\xi} = A_q \xi + b_q,
\end{equation}
where $q \in \mathcal{Q}_{\text{tr}} \coloneqq \{1, \dots, l\}$ denotes the current mode of the system, $A_q\in\mathbb{R}^{n\times n}$, $E_q \in \mathbb{R}^{n\times n}$, $b_q \in \mathbb{R}^n$, and $\xi \in \mathbb{R}^n$. We make the following regularity assumption on system \eqref{eqn: affine_DAE}:
\begin{assump}\label{assumpregdae}
For each mode $q \in \mathcal{Q}_{\text{tr}}$, the matrix $A_q$ is invertible. \hfill $\square$
\end{assump}
Under Assumption \ref{assumpregdae}, for each fixed mode $q \in \mathcal{Q}_{\text{tr}}$, system \eqref{eqn: affine_DAE} has a unique equilibrium point given by $\xi^*_q = -A_q^{-1}b_q$. Additionally, under Assumptions~\ref{assumpregdae} and
\ref{assump: regularity}, and the average dwell-time switching
bound~\eqref{ADTeq}, the switched DAE~\eqref{eqn: affine_DAE}
can be represented as the following hybrid dynamical system $\mathcal{H}$ with data $(C,F,D,G)$ and state vector $x = (\xi, q, \tau)$: 
\begin{subequations}
    \label{eqn: eqiv_HI}
    \begin{align}
        \begin{bmatrix}
            \Dot{\xi} \\
            \Dot{q} \\
            \Dot{\tau}
        \end{bmatrix} &\in F(x):= 
        \begin{bmatrix}
            \Pi^{\text{diff}}_{q} (A_q \xi + b_q) \\
            0 \\
            [0, \frac{1}{\tau_d}]
        \end{bmatrix},~x \in C \\[2mm]
        \begin{bmatrix}
            \xi^+\\ 
            q^+ \\
            \tau^+
        \end{bmatrix} &\in G(x):= 
        \begin{bmatrix}
        \left\{\begin{array}{c}\Pi_{z}\xi\\ z\end{array}:z\in\mathcal{Q}_{\text{tr}}\backslash\{q\}\right\}\\
            \tau - 1
        \end{bmatrix},~x \in D
    \end{align}
    where $\Pi_q^{\mathrm{diff}}$ is the differential projector and
$\Pi_z$ is the affine consistency projector; see
Definitions \ref{defn: CP} and \ref{defn: diff_proj} in Appendix~\ref{Appendix:auxdef}.
    
    To define the flow and jump sets of the system, we introduce the \emph{consistency set} of the $q^{th}$ mode of the DAE: 
   \begin{equation}\label{consistencyset}
   \mathcal{C}_q:=  \{\xi \in \mathbb{R}^n: \xi \in (\text{span}(V_q) - W_q b^a_q)\}.
 \end{equation}
In \eqref{consistencyset}, the matrices $V_q \in \mathbb{R}^{n \times n_1^q}$ and $W_q \in \mathbb{R}^{n \times n^q_2}$ define the so-called Quasi-Weirstrass form of the pair $(E_q,A_q)$ (see Lemma \ref{lem: QWF} in Appendix \ref{Appendix:auxdef}), given by:
 \begin{equation}\label{QWform1}
        (S_q E_q T_q, S_q A_q T_q) = \left ( \begin{bmatrix}
            I_{n^q_1} & 0 \\
            0 & N_q
        \end{bmatrix}, \begin{bmatrix}
            J_q & 0 \\
            0 & I_{n^q_2}
        \end{bmatrix} \right ),
    \end{equation}
where $S_q \coloneqq [E_qV_q, A_qW_q]^{-1}$, $T_q:=[V_q,W_q]$, $J_q \in \mathbb{R}^{n_1^q \times n^q_1}$, and  $N_q \in \mathbb{R}^{n_2^q \times n_2^q}$ is a nilpotent matrix with $(N_q)^{n_2^q} = 0$, $n_1^q + n_2^q = n$.  The following assumption is needed to guarantee the regularity of the solutions to \eqref{eqn: affine_DAE}:
\begin{assump}
\label{assump: regularity}
For each $q \in \mathcal{Q}_{\text{tr}}$,  the pair $(E_q,A_q)$ is \emph{regular}, i.e., the characteristic polynomial $\text{det}(sE_q -A_q)$ is not identically zero. Moreover, the Quasi-Weierstrass form \eqref{QWform1} of the pair $(E_q, A_q)$ is such that the nilpotent block $N_q$  has index at most one, i.e., $N_q = 0$. 
\hfill \qed
\end{assump}
\begin{rem}\normalfont{
A nonzero characteristic polynomial guarantees the quasi-Weierstrass decomposition \eqref{QWform1}, but may allow higher-index nilpotent blocks and thus physically meaningless impulsive solutions. Assumption~\ref{assump: regularity} rules these out and ensures well-defined solutions on the consistency set; see Lemma~\ref{lem: QWF}. For the power-system model in Section~\ref{sec:power}, the algebraic equations are explicit, yielding $N_q=0$, so the assumption holds automatically.}\hfill \qed
\end{rem}

Using the consistency set \eqref{consistencyset}, the flow and jump sets of the HDS are defined as follows: 
    \begin{align}
    C &:= \bigcup_{q \in \mathcal{Q}_{\text{tr}}} C_q,~~~~C_q\coloneqq \mathcal{C}_q \times \{q\} \times [0,N_0],\\
        D &:= \bigcup_{q \in \mathcal{Q}_{\text{tr}}} D_q,~~~~D_q:= \mathcal{C}_q \times \{q\} \times [1, N_0]. \label{jumpset01}
    \end{align}

\end{subequations} 
In the resulting HDS \eqref{eqn: eqiv_HI}, the state $\tau$ acts as a timer that regulates the frequency of the jumps. In particular, by the structure of \eqref{jumpset01}, jumps are allowed only when $\tau\geq 1$. Each jump triggers a switch on the value of the mode $q$ that parametrizes the continuous-time dynamics of the system. Since the maximum rate of increase of $\tau$ is limited by the parameter $1/
\tau_d$, by the structure of the flow and jump set, the switching frequency in the resulting HDS satisfies the ADT constraint, see \cite[Example 2.15]{teel2013lyapunov}. This mechanism rules out Zeno behavior, but allows a finite number of arbitrarily fast switches. Note that, in general, the equilibrium point of the dynamics $\dot{\xi}$ is mode-dependent.  We make the following stability assumption on each mode:
\begin{assump}
    \label{assump: matrices}
For each $q \in \mathcal{Q}_{\mathrm{tr}}$, there exist symmetric matrices
$P_q,M_q \in \mathbb{R}^{n\times n}$ that are positive definite on
$\operatorname{span}(V_q)$ and satisfy
\[
\eta^\top
\left(
A_q^\top \Pi_q^{\mathrm{diff}\top} P_q
+
P_q \Pi_q^{\mathrm{diff}} A_q
\right)\eta
\leq
-\eta^\top M_q \eta,
\]
for all $\eta \in \operatorname{span}(V_q)$. \hfill \qed
\end{assump}
In essence, Assumption~\ref{assump: matrices} requires that, for each mode $q\in\mathcal{Q}_{\mathrm{tr}}$, the linear component of the $\xi$-dynamics satisfies a standard mode-dependent Lyapunov inequality. Note that, since the consistency set $\mathcal C_q$ is in general affine, the shifted state
$\tilde{\xi}=\xi-\xi_q^\ast$ (used in the stability analysis) does not evolve on
$\mathcal C_q$. Rather, since both $\xi$ and $\xi_q^\ast$
belong to $\mathcal C_q$, we have $\tilde{\xi}\in
\mathcal C_q-\mathcal C_q
=
\operatorname{span}(V_q)$.
Consequently, all positive-definiteness and quadratic-form conditions involving the shifted state are imposed on
$\operatorname{span}(V_q)$, the linear direction space of $\mathcal C_q$.

\subsection{Global Uniform Ultimate Boundedness under Average-Dwell Time Switching}
\label{subsec:sgpas}
To analyze the stability properties of system \eqref{eqn: eqiv_HI}, we introduce the change of variables $x \mapsto (\Tilde{\xi}, q, \tau) =:y$ where $\Tilde{\xi} = \xi -\xi^*_q\in\texttt{span}(V_q)$ and $\xi^*_q $ is the \emph{unique} point in $\mathcal{C}_q$ that satisfies $A_q \xi^*_q + b_q = 0$. Such a point exists due to Assumption \ref{assumpregdae}. For each $q\in\mathcal{Q}_{tr}$, the dynamics of $\Tilde{\xi}$ during flows are
\begin{align*}
    \Dot{\Tilde{\xi}} & =  \Pi_q^{\text{diff}} (A_q \xi + b_q)=  \Pi_q^{\text{diff}} (A_q (\Tilde{\xi} + \xi_q^*) + b_q)= \Pi_q^{\text{diff}} A_q \Tilde{\xi}.
\end{align*}
During jumps, using Definition \ref{defn: CP}:
\begin{align*}
    \Tilde{\xi}^+ 
    = L_{q^+}\tilde{\xi} + \Delta_{q,q^+}. 
\end{align*}
where $\Delta_{q,q^+}=L_{q^+}\xi^*_q + d_{q^+} - \xi^*_{q^+}$.
 {Note that the residuals $\Delta_{q,z}$ vanish for every admissible transition
$q\to z$ if and only if $\Pi_z\xi_q^*=\xi_z^*$ for every
$q,z\in\mathcal{Q}_{\mathrm{tr}}$ with $q\neq z$. This condition expresses
compatibility of the  equilibria of the modes under the consistency projections; it does not, in general, require the equilibrium vectors $\xi_q^*$ and
$\xi_z^*$ to coincide.}

In the $y$-coordinates, the system is described by the following HDS:
\begin{subequations}
    \label{eqn: equiv_2}
    \begin{align}
    &\dot{\Tilde{\xi}}= \Pi^{\text{diff}}_{q} A_q \Tilde{\xi},~~~\dot{q}=0,~~~\dot{\tau}\in\left[0, \frac{1}{\tau_d}\right]
        ~~~y \in C, \\[2mm]
        &\begin{bmatrix}
            \Tilde{\xi} ^+\\
            q^+\\
            \tau^+
        \end{bmatrix}\in 
        \begin{bmatrix}
        \left\{\begin{array}{c}L_{z}\Tilde{\xi} + \Delta_{q,z}\\z\end{array}:z\in\mathcal{Q}_{\text{tr}}\backslash\{q\}\right\}\\
            \tau - 1
        \end{bmatrix},~~y \in D,
    \end{align}
\end{subequations}
where $\mathcal{C}_q$ is substituted by $\text{span}(V_q)$ in the definitions of $C$ and $D$. To study the stability properties of system \eqref{eqn: equiv_2}, we define the following compact set: 
\begin{equation}
    \label{eqn: set_A}
    \mathcal{A} \coloneqq \{0\}  \times \mathcal{Q}_{\text{tr}}\times [0,N_0],
\end{equation}
as well as the following constants:
\begin{subequations}
    \begin{align}
        \lambda &\coloneqq \frac{\min_{q \in  {\mathcal{Q}_{\mathrm{tr}}}} \lambda_{\text{min}} (O^\top_q M_q O_q)} {\max_{q \in  {\mathcal{Q}_{\mathrm{tr}}}} \lambda_{\text{max}}(O^\top_q P_q O_q)}, \\
        \gamma & \coloneqq \frac{\max_{q \in  {\mathcal{Q}_{\mathrm{tr}}}}\lambda_{\text{max}} (O_{q}^\top P_{q} O_{q}) \cdot  \max_{q \in  {\mathcal{Q}_{\mathrm{tr}}}}\lVert \textcolor{black}{L_{q}} \rVert_2^2}{\min_{q \in  {\mathcal{Q}_{\mathrm{tr}}}}\lambda_{\text{min}} (O_q^\top P_q O_q)}\\
        \alpha & \coloneqq \max_{q\in\mathcal{Q}_{\text{tr}}}\max_{r\in\mathcal{Q}_{\text{tr}}\setminus\{q\}}
\Delta_{q,r}^\top P_r \Delta_{q,r}\\
\underline{p}&:= \min_{q\in {\mathcal{Q}_{\mathrm{tr}}}}\lambda_{\min}(O_q^\top P_q O_q),~\overline{p}:= \max_{q\in {\mathcal{Q}_{\mathrm{tr}}}}\lambda_{\max}(O_q^\top P_q O_q).
    \end{align}
\end{subequations}
By Lemmas \ref{lem:PD} and \ref{lem: eig_vals} in the Appendix, we have that
 $\underline{p},\overline{p},\lambda,\gamma>0$ and $\alpha\geq0$. The following theorem is the first main result of this section. The proof is presented in Appendix \ref{appen:proofs}.
\begin{thm}
\label{thm:ISS_like}
Suppose that Assumptions~\ref{assumpregdae}-\ref{assump: matrices} hold.
Choose any $ {\mu>\max\{0,\ln\gamma\}}$ and any
$\theta\in(0,e^\mu/\gamma-1)$, and suppose that
$\mu/\lambda<\tau_d$. Let
$\kappa:=e^{-\mu}(1+\theta)\gamma$, $\eta:=\min\left\{
\lambda-\frac{\mu}{\tau_d},-\ln\kappa
\right\}$. Then, every solution of system~\eqref{eqn: equiv_2} satisfies the following bound for all $(t,j) \in \mathrm{dom}(y)$:
\begin{equation}\label{eq: kl_bound}
    |y(t,j)|_{\mathcal{A}} \;\leq\; \beta\!\left(|y(0,0)|_{\mathcal{A}},\, t+j\right) + \nu(\alpha),
\end{equation}
where $\beta(r,s):= (\frac{{\rm e}^{\mu N_0}\,\overline{p}}{\underline{p}})^{0.5}\rm e^{-\eta s/2}\, r$, and $\nu(\alpha):=\sqrt{(\bar{c}/\underline{c})\alpha}$,  $\bar{c}={\rm e}^{\mu N_0}\bigl(1+\tfrac{1}{\theta}\bigr)$, $\underline{c}=\underline{p}\,\bigl({\rm e}^{\mu} - (1+\theta)\gamma\bigr)$.
\hfill $\square$
\end{thm}
\begin{rem}\normalfont{
Theorem~\ref{thm:ISS_like} gives a lower bound on $\tau_d$, and thus on the switching rate, guaranteeing global practical stability of \eqref{eqn: equiv_2} with respect to $\mathcal{A}$. The residual neighborhood $\nu(\alpha)$ depends on the maximum projected-equilibrium mismatch $\alpha$, rather than the distance between the original equilibria. In particular, $\alpha=0$ iff $\Pi_r(\xi_q^*)=\xi_r^*$ for all admissible transitions, in which case $\mathcal{A}$ is UGAS. This includes the common-equilibrium setting of~\cite{trenn2012switched}, while also allowing mode-dependent equilibria compatible under the consistency projections.
} \hfill \qed
\end{rem}
\subsection{Characterization of the Limiting Behavior of the Switched DAE: $\Omega$-Limit Sets}
While Theorem \ref{thm:ISS_like} provides a $\mathcal{K}\mathcal{L}$ bound on the trajectories of the system that holds for all time, provided $\tau_d$ is lower-bounded by the ratio $\mu/\lambda$, a sharper characterization of the limiting behavior of the HDS as $t+j\to\infty$ can be obtained by characterizing its associated $\Omega$-limit whenever $\tau_d$ is ``sufficiently large'' and no two consecutive switches are allowed, i.e., $\tau_d\to\infty$ and $N_0=1$. To study this case, for each $q \in \mathcal{Q}_{\text{tr}}$, let $\Theta_q^t(\xi_0)$ denote the solution to the DAE $E_q\Dot{\xi} = A_q \xi + b_q$ starting from a consistent initial condition $\xi_0 \in \mathcal{C}_q$, and evaluated at time $t$. Using \cite[Theorem 6. 4.4]{trenn2012switched}, it follows that
\begin{align}
    \label{eqn: dae_sol}
    \Theta_q^t(\xi_0) = &e^{A^{\text{diff}}_{q} t} (\xi_0 + \Pi_q^\text{imp}b_q) + \\ 
    &\int_0^t e^{A^\text{diff}_{q}(t-s)} \nonumber \Pi_q^\text{diff} b_q \, ds - \Pi_q^\text{imp} b_q, 
\end{align}
where $A_q^{\text{diff}} = \Pi^\text{diff}_{q} A_q$. Using the form of \eqref{eqn: dae_sol}, as well as Definitions \ref{definitionReachable} and \ref{defn:omega_set}, we now state the second main result of this section. The proof is presented in the Appendix.
\begin{thm}
    \label{thm: SGPAS}
    Let $N_0=1$ and suppose that Assumptions \ref{assumpregdae},  \ref{assump: regularity}, and \ref{assump: matrices} hold. Let 
\begin{align}
\mathcal{A}
:=
\bigcup_{q\in\mathcal{Q}_{\mathrm{tr}}}
\Bigg[
&\{\xi_q^\star\}\times\{q\}\times[0,1]\notag
\\
&{}\cup
\bigcup_{\substack{r\in\mathcal{Q}_{\mathrm{tr}}\\ r\neq q}}
\overline{
\left\{
\Theta_r^t\bigl(\Pi_r(\xi_q^\star)\bigr)
:
t\geq 0
\right\}}
\times\{r\}\times\{0\}
\Bigg].
\end{align}
    %
    Then, the set $\mathcal{A}$ is SGPAS as $1/\tau_d \to 0^+$ for system \eqref{eqn: eqiv_HI}.
    %
\end{thm}
\begin{rem}\normalfont{
Compared with Theorem~\ref{thm:ISS_like}, Theorem~\ref{thm: SGPAS} more precisely characterizes the asymptotic behavior for sufficiently large $\tau_d$. In particular, $\mathcal Y:=\operatorname{proj}_{\xi}(\mathcal A)$ consists of the mode-dependent equilibria and the trajectories flowing from $\Pi_r(\xi_q^\star)$ toward $\xi_r^\star$, for $r\neq q$. Theorem~\ref{thm: SGPAS} guarantees convergence to an arbitrarily small neighborhood of $\mathcal Y$.
}\hfill\qed
\end{rem}

Before applying Theorems~\ref{thm:ISS_like} and~\ref{thm: SGPAS} to the power transmission system introduced in Section~III, we illustrate the main theoretical results through two low-dimensional numerical examples in $\mathbb{R}^2$ and $\mathbb{R}^3$, respectively. The first example is adapted from~\cite[Example~7]{nanez2017invariance} to account for mode-dependent equilibria.
\begin{ex}
\label{ex: example_7_nanez}
Consider the switched DAE with data $E_1=[0,1;0,0]$, $A_1 = [0, -1;1, -1]$, $b_1 = [1,0]$, $E_2=[0,0;0,1]$, $A_2=[1, 0;0, -1]$, $b_2 =[2,-2]$. The dynamics of mode $1$ are given by $\Dot{\xi}_2= -\xi_2 + 1$, $\xi_1= \xi_2$. Similarly, mode 2 has dynamics given by  $\Dot{\xi}_2= -\xi_2 - 2$ $\xi_1 = -2$. The equilibria of the two subsystems are $\xi^*_1 =[1, 1]$
 and $\xi^*_2=[-2, -2]$, respectively. 
%
%
The consistency sets of the two modes are $\mathcal{C}_1=\left\{[s;s]:s\in\mathbb{R}\right\}=
    \operatorname{span}
    \left\{[1;1]
    \right\}$,
        $\mathcal{C}_2=
    \left\{[-2;s]:s\in\mathbb{R}\right\}
    =[-2;0]+
    \operatorname{span}
    \left\{[0;1]\right\}$. Using the quasi-Weierstrass form, the corresponding affine
consistency projectors $\Pi_q(\xi)=L_q\xi+d_q$ are $\Pi_1(\xi)=[0,1;0,1]\xi=[\xi_2;\xi_2]$, $\Pi_2(\xi)=[0,0;0,1]\xi+[-2;0]$=$[-2;\xi_2]$. Thus, $L_1=[0,1;0,1]$, $L_2=[0,0;0,1]$, $d_1=0$, $d_2=[-2;0]$.
Similarly, the differential projectors are given by  $\Pi^{\text{diff}}_1 =[1,0;1,0]$, $\Pi^{\text{diff}}_2 =[0, 0;0,1]$. A simple computation shows that the matrices $P_1 = [1, 0; 0, 0]$ and $P_2 =[0, 0; 0, 1]$ are positive definite on $\text{span}(V_1)$ and $\text{span}(V_2)$ respectively. Choosing $O_1 = \frac{1}{\sqrt{2}} \begin{pmatrix}
    1 & 1
\end{pmatrix}$ and $O_2 = \begin{pmatrix}
    0 & 1
\end{pmatrix}$, yet another computation shows that the matrices $M_1 =[0,1;1,0]$ and $M_2 =[0, 0; 0, 2]$ are positive definite on
$\operatorname{span}(V_1)$ and $\operatorname{span}(V_2)$,
respectively. From these matrices, the value of $\lambda$ and $\gamma$ are $1$ and $4$, respectively. We choose $N_0 = 2$ here. 
To satify the conditions of Theorem \ref{thm:ISS_like} we take $\mu = 2$, $\theta = 0.35 \in (0, e^{\mu}/\gamma -
1)$, $\frac{1}{\tau_d} = 0.4 < \lambda/\mu = 0.5$. This implies that 
$\kappa = e^{-2}(1.35)(4) \approx 0.731$ which lies in the interval $(0,1)$ and
$\eta = \min\{\lambda - \mu\cdot \sfrac{1}{\tau_d},\, -\ln\kappa\}
= \min\{0.2,\, 0.313\} = 0.2$.
Figure \ref{fig:thm1} shows the trajectories of the system, illustrating the result of Theorem \ref{thm:ISS_like}. \hfill \qed
\end{ex}
\begin{figure}[t!]
    \centering
    \includegraphics[width=0.99\linewidth]{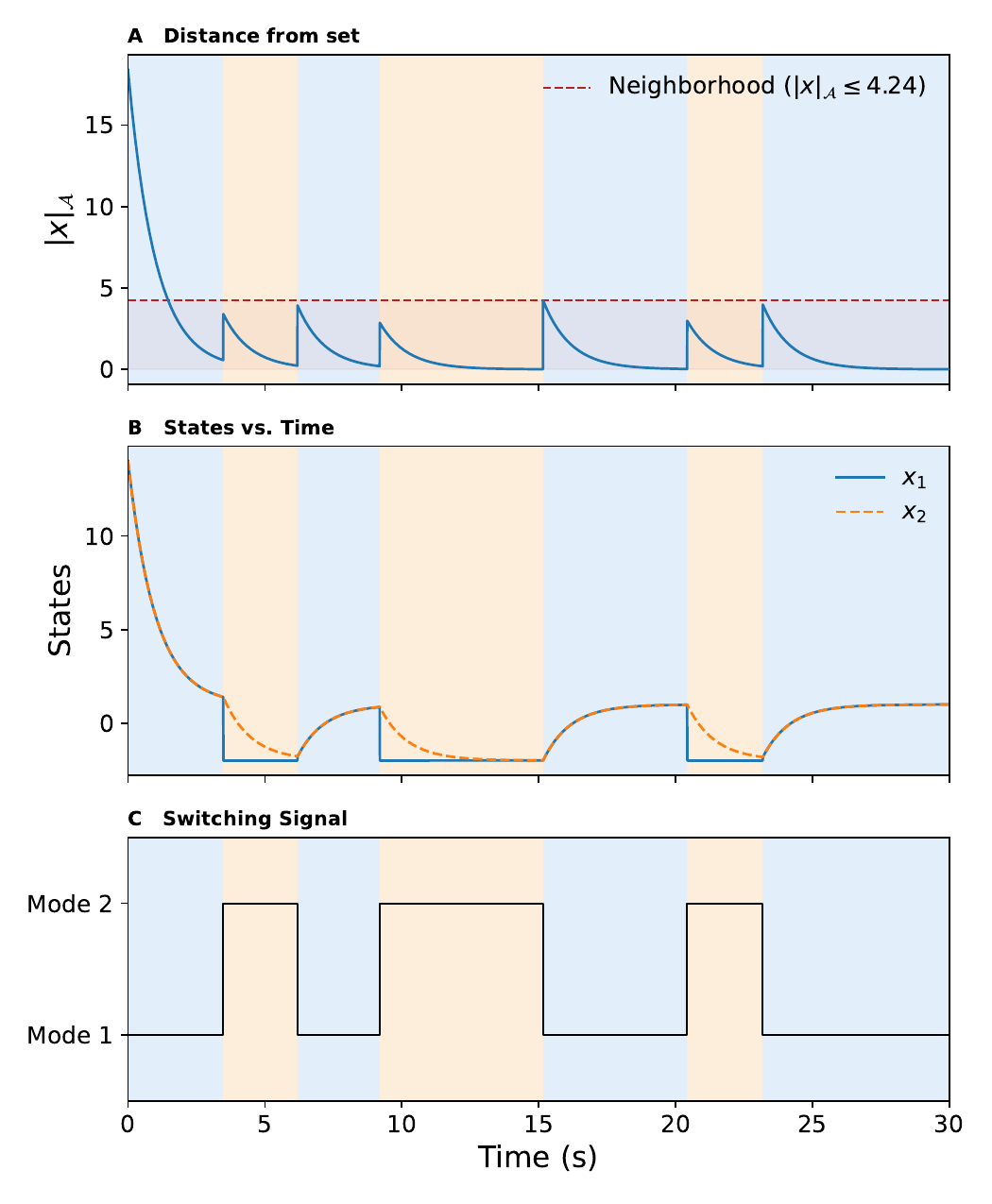}
    \caption{Example~\ref{ex: example_7_nanez} illustrating Theorem~\ref{thm:ISS_like}: (A) distance $|x|_{\mathcal A}$, (B) state trajectories $x_1,x_2$, and (C) switching signal.} 
    \label{fig:thm1}
\end{figure}
\begin{ex}
\label{ex:spiral}
To illustrate the set $\mathcal{Y}$, consider a two-mode system in $\mathbb{R}^3$ with $\xi_3$ as the algebraic
variable. The dynamics of the first mode are given by $\Dot{\xi}_1= -\xi_1 - 2\xi_2$, $\Dot{\xi}_2= \phantom{-}2\xi_1 - \xi_2$, $0=\xi_3$. The second mode has dynamics $\Dot{\xi}_1= -\xi_1 + 3\xi_2 + 1$, $\Dot{\xi}_2= -3\xi_1 - \xi_2 + 3$, $0 = \xi_3 - 1$. The consistency spaces $\mathcal{C}_1=\{\xi_3=0\}$ and
$\mathcal{C}_2=\{\xi_3=1\}$ are parallel horizontal planes in
$\mathbb{R}^3$, separated by the offset in $b^q$.
In standard form, both modes share $E_1 = E_2 =[1,0,0;0,1,0;0,0,0]$ with $b_1=(0,0,0)$, $b_2=(1,3,-1)$, and $A_1=[-1,-2,0;2,-1,0;0,0,1]$, $A_2=[-1,3,0;-3,-1,0;0,0,1]$.
%
%
The differential projectors are identical for both modes, and the
consistency projectors differ only in their affine offset: $\Pi_1^{\mathrm{diff}}=\Pi_2^{\mathrm{diff}} =[1,0,0;0,1,0;0,0,0]$, $\Pi_1(\xi)=[\xi_1,\xi_2,0]$, $\Pi_2(\xi) =[\xi_1,\xi_2,1]$, with equilibria $\xi_1^* = (0,0,0) \in \mathcal{C}_1$ and
$\xi_2^* = (1,0,1) \in \mathcal{C}_2$. The set
$\mathcal{Y}=\mathcal{Y}_1\cup\mathcal{Y}_2$
predicted by Theorem~\ref{thm: SGPAS} consists of the two
closed spiral orbits $\mathcal{Y}_1=
\overline{
\left\{
\Theta_2^t\bigl(\Pi_2(\xi_1^\star)\bigr)
:
t\geq0
\right\}}$, and $\mathcal{Y}_2=
\overline{
\left\{
\Theta_1^t\bigl(\Pi_1(\xi_2^\star)\bigr)
:
t\geq0
\right\}}$. Figure~\ref{fig:omega_limit} shows the structure of
$\mathcal{Y}$ which we recall, is the projection of the set $\mathcal{A}$ onto the $\xi$ coordinates, together with the equilibria and consistency projectors.
Figure~\ref{fig:trajectories} shows system trajectories
converging to $\mathcal{Y}$. \hfill \qed
\end{ex}
\begin{figure}[t!]
    \centering
    \includegraphics[width=0.97\linewidth]{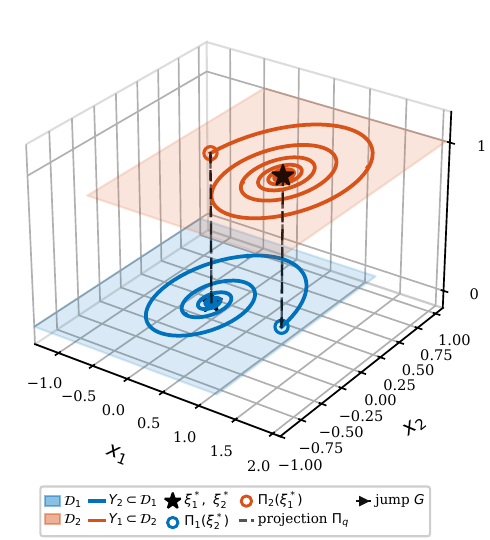}
    \caption{Set $\mathcal{Y}=\mathcal{Y}_1\cup\mathcal{Y}_2$ for Example~\ref{ex:spiral}, projected onto $(x_1,x_2)$. The curves connect the projected equilibria $\Pi_q(\xi_{q'}^*)$ to $\xi_q^*$, with dotted arrows indicating the consistency projections.}
    \label{fig:omega_limit}
    \vspace{-0.4cm}
\end{figure}
\begin{figure}[t!]
    \centering
    \includegraphics[width=0.99\linewidth]{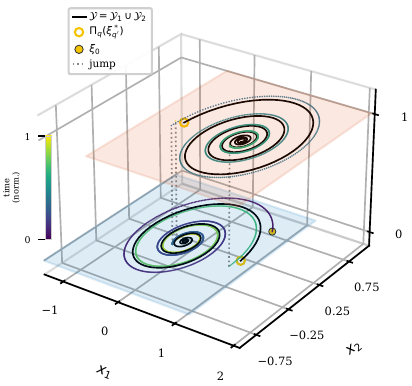}
    \caption{Trajectories of Example~\ref{ex:spiral} illustrating Theorem~\ref{thm: SGPAS}(a) in the $\xi$-coordinates. Solutions spiral on the parallel consistency planes $\mathcal{C}_1=\{x_3=0\}$ (blue) and $\mathcal{C}_2=\{x_3=1\}$ (red), with dotted segments indicating switching-induced jumps between them. The color gradient denotes normalized time, and trajectories accumulate on $\mathcal{Y}=\mathcal{Y}_1\cup\mathcal{Y}_2$ (black).}
    \label{fig:trajectories}
\end{figure}
\begin{figure}[h!]
    \centering
    \includegraphics[width=0.9\linewidth]{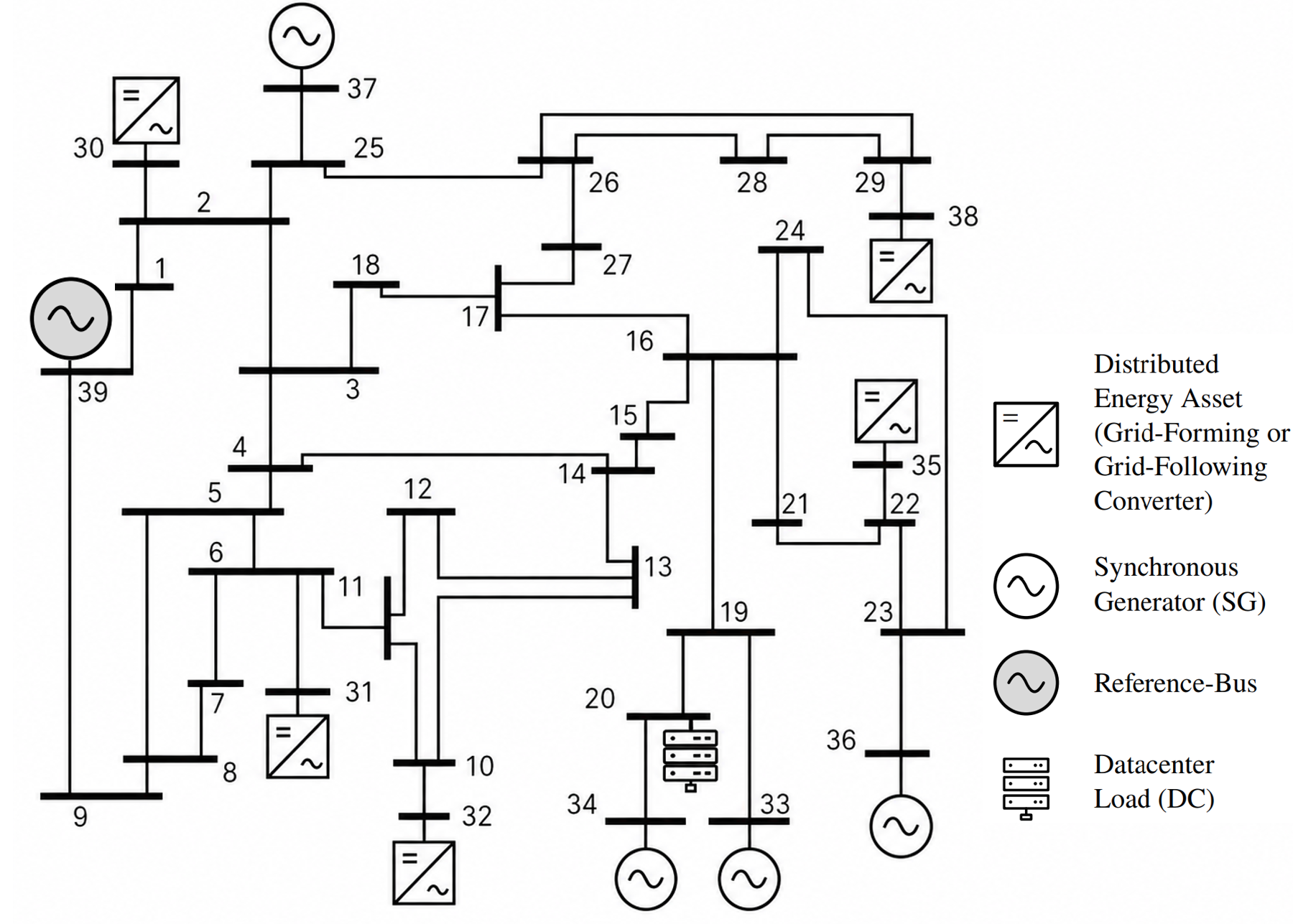}
    \caption{Modified IEEE 39-bus system used for training and inference simulations, with Bus~39 as the reference and the datacenter load at Bus~20.}
    \label{fig:ieee_modified}
    \vspace{-0.4cm}
\end{figure}

\subsection{Application to AI Power Transmission Systems with Training Workloads: Simulations and EMT Validations}
\begin{figure*}[t]
    \centering

    \begin{subfigure}[t]{0.49\textwidth}
        \centering
        \includegraphics[width=\textwidth]{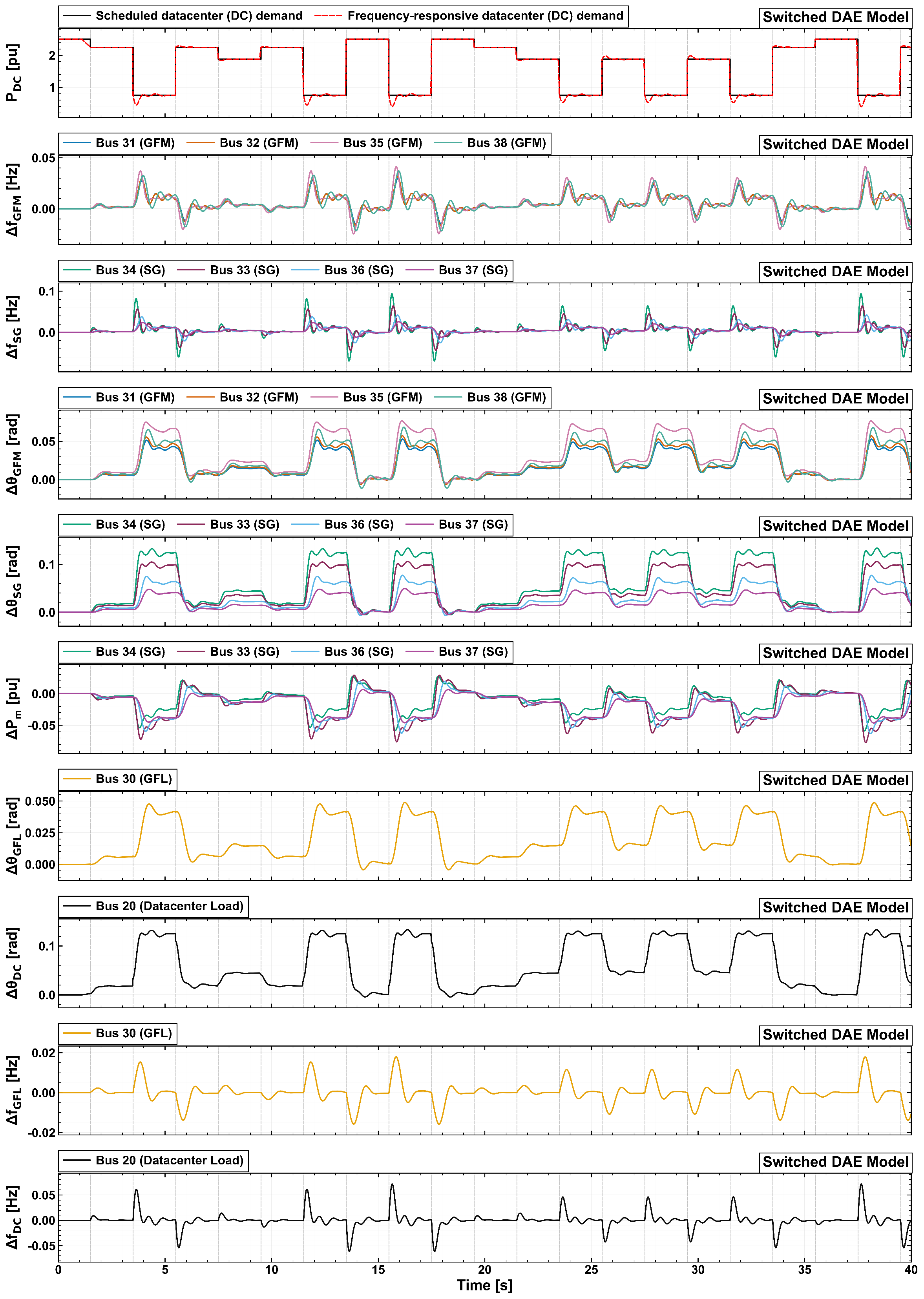}
        \caption{Deterministic Switched DAE model}
        \label{fig:deterministic_dae}
    \end{subfigure}
    \hfill
    \begin{subfigure}[t]{0.49\textwidth}
        \centering
        \includegraphics[width=\textwidth]{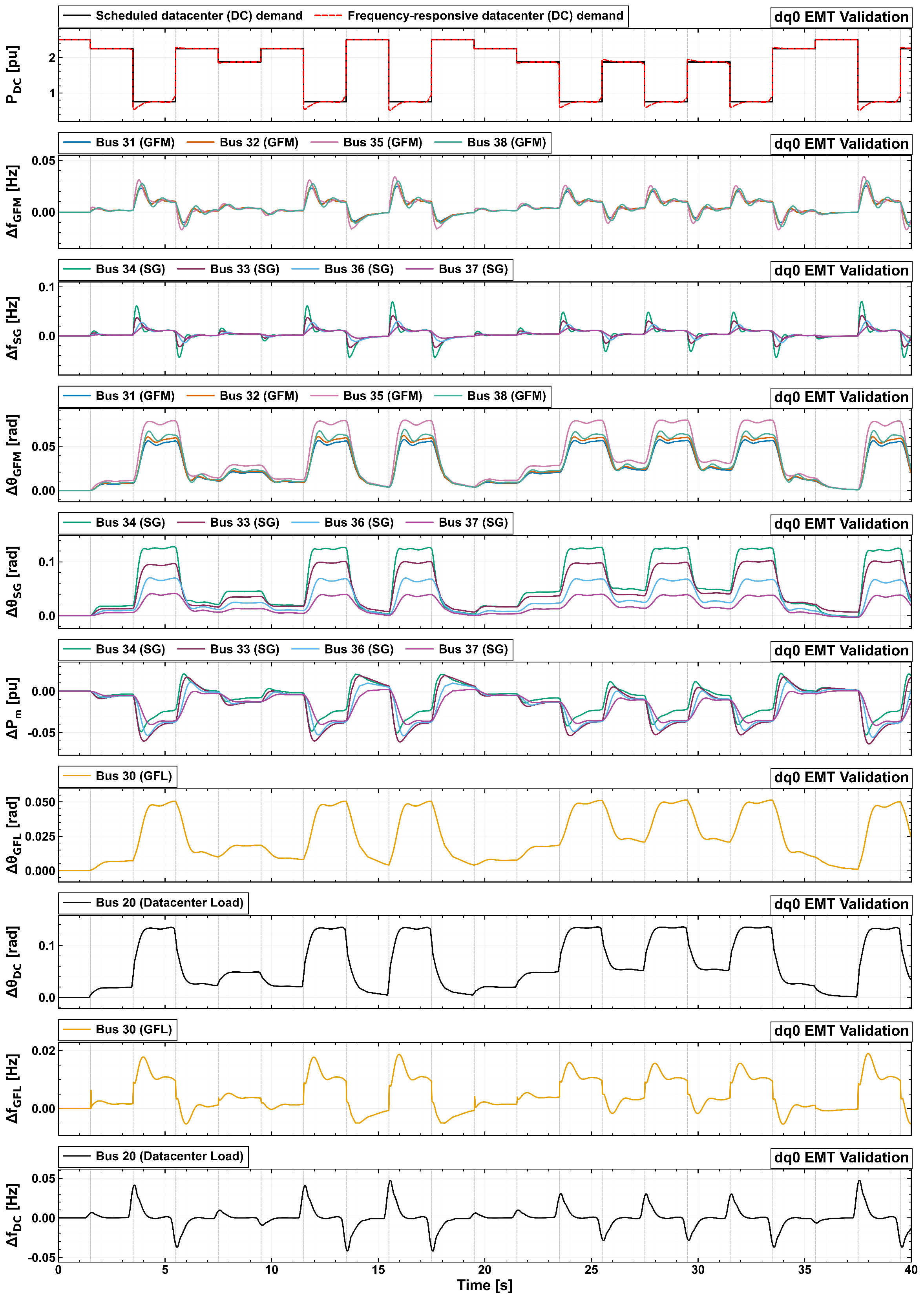}
        \caption{dq0 EMT validation}
        \label{fig:deterministic_emt}
    \end{subfigure}

    \vspace{-0.3em}
    \caption{Deterministic training-phase switching on the modified IEEE 39-bus system with AI datacenter loads: (a) switched-DAE model and (b) high-fidelity dq0 EMT validation in PSID.jl, both under switching over $\mathcal{Q}_{\text{tr}}=\{fw,bw,ar,id\}$.}
\label{fig:switched_dae_emt_validation_responses}
\vspace{-0.3cm}
\end{figure*}
To study the role of Theorems~\ref{thm:ISS_like} and~\ref{thm: SGPAS} in the power transmission system described in Section~III, we consider a modified New England IEEE 39-bus transmission network, shown in Figure~\ref{fig:ieee_modified}, incorporating heterogeneous generation resources and an AI datacenter load. Specifically, the network hosts five synchronous generators at Buses~33, 34, 36, 37, and 39, with Bus~39 serving as the reference bus; four GFM IBR units at Buses~31, 32, 35, and 38; and one GFL IBR unit at Bus~30. A frequency-responsive datacenter load is connected at Bus~20, while the GFM unit at Bus~38 is designated as the campus GFM and provides frequency support during variations in datacenter demand. Numerical studies are
presented for deterministic
training-phase switching over the set of modes $\mathcal{Q}_{\text{tr}}=\{fw,bw,ar,id\}$. In our numerical experiments, we consider two different cases:

\vspace{0.2cm}\noindent 
\emph{Case I: Stability under ADT Switching:} To study stability under general ADT switching, we set $N_0=2$ and let the datacenter at Bus~20 cycle through $p_{dc}^{fw}>p_{dc}^{bw}>p_{dc}^{ar}>p_{dc}^{id}$, with frequency support from the GFM inverter at Bus~38. We simulate the switched-DAE model of Section~\ref{sec:power} and validate it against high-fidelity dq0 electromagnetic transient (EMT) simulations in PowerSimulationsDynamics.jl (PSID.jl)~\cite{lara2023powersimulationsdynamics}. The EMT model captures network, converter, and electromechanical dynamics neglected by the reduced-order model, providing a benchmark for its accuracy under fast datacenter workload transitions.

\begin{figure*}[h]
    \centering
        \includegraphics[width=\textwidth]{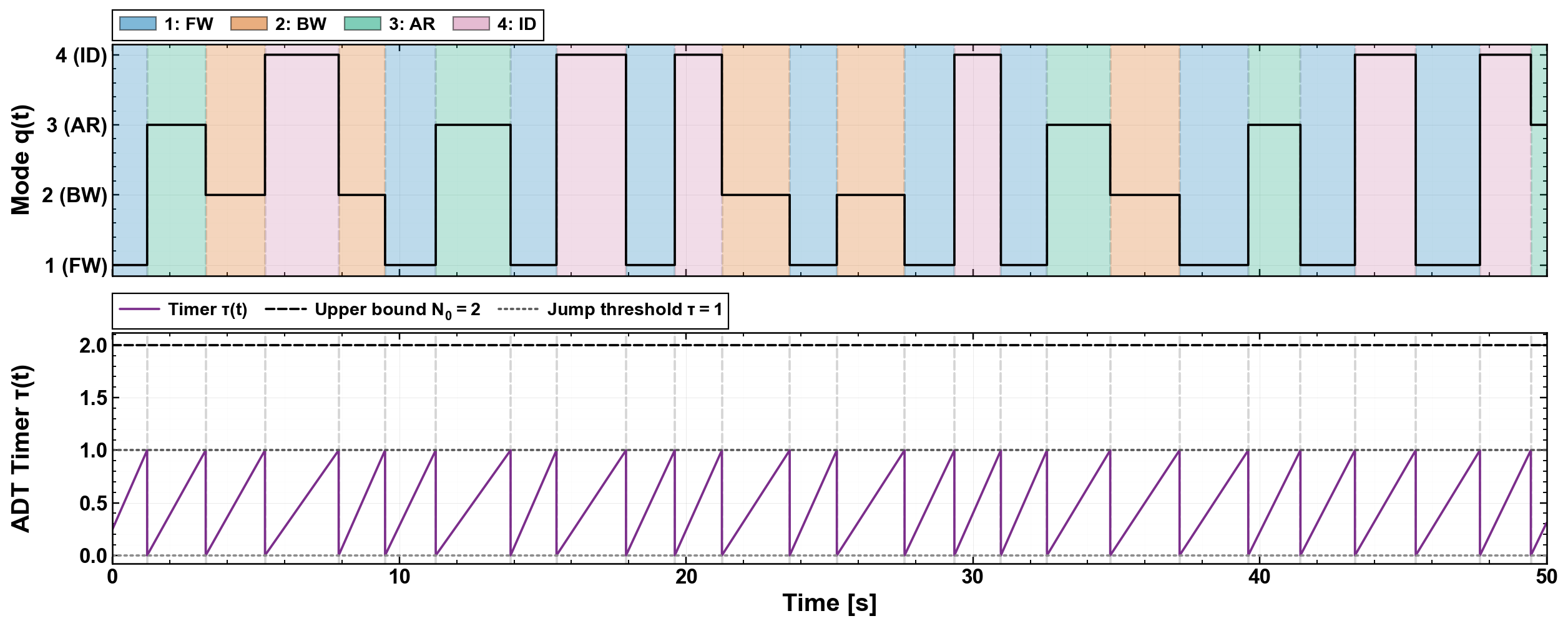}
\label{fig:deterministic_training_switching_timer}
    \vspace{-0.3em}
    \caption{Deterministic training-phase switching signal and ADT timer, for the scenario in Fig.~\ref{fig:switched_dae_emt_validation_responses}. The figure shows the switching signal over $\mathcal{Q}_{\text{tr}}=\{fw,bw,ar,id\}$ and the timer $\tau(t)$, which resets at each mode transition and imposes the jump threshold $\tau=1$ under the bound $N_0=2$.}
\label{fig:deterministic_switching_and_lyapunov}
\vspace{-0.5cm}
\end{figure*}

The PSID.jl model uses the same modified IEEE 39-bus network, operating point, generation mix, datacenter location, and mode transitions as the switched-DAE model, differing primarily in fidelity. It includes detailed synchronous-generator, converter, control, PLL, and transmission-line dynamics, while the datacenter at Bus~20 follows the same frequency-responsive load model~\eqref{eq:dcbal}.

\begin{figure*}[h]
    \centering 
    \begin{subfigure}[t]{0.49\textwidth}
    \includegraphics[width=0.85\linewidth]{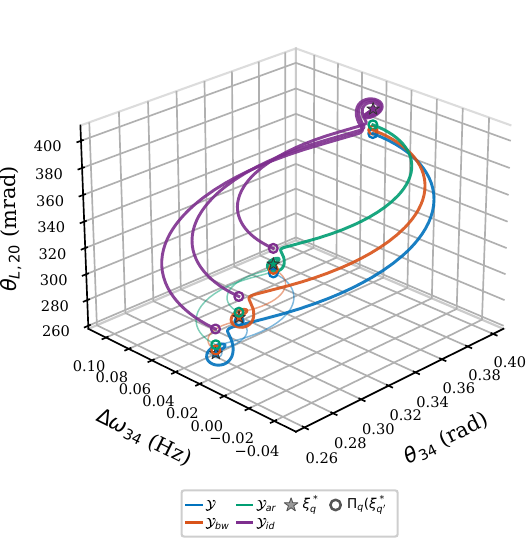}
    \caption{}
    \label{fig:ols_39}
    \end{subfigure}
    \hfill
    \begin{subfigure}[t]{0.49\textwidth}
    \centering  \includegraphics[width=0.9\linewidth]{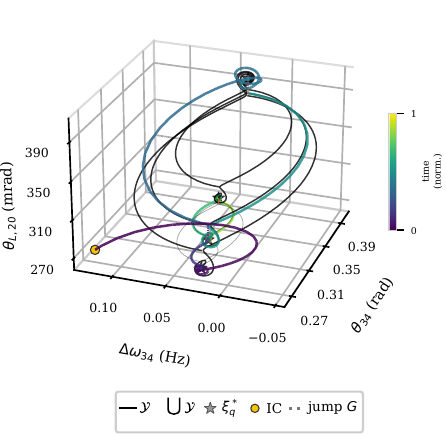}
    \caption{}
    \label{fig:thm2_39}
    \end{subfigure}
    \caption{(a) Projection of $\mathcal{Y}$ for the modified IEEE 39-bus system under training-phase switching, showing mode equilibria $\xi_q^*$ (stars), consistency projections $\Pi_q(\xi_{q'}^*)$ (circles), and limiting orbits (black). (b) Trajectories in $(\theta_{34},\Delta\omega_{34},\theta_{L,20})$ illustrating Theorem~\ref{thm: SGPAS}(a), including switching-induced jumps and convergence toward $\mathcal{Y}$. Color denotes normalized time.}
    \vspace{-0.4cm}
\end{figure*}

Figure~\ref{fig:switched_dae_emt_validation_responses} compares the switched-DAE model (Panel a) with the high-fidelity dq0 EMT simulation (Panel b), showing datacenter power demand and the relevant frequency, angle, mechanical-power, and bus responses. Figure~\ref{fig:deterministic_switching_and_lyapunov} additionally shows the active mode $q\in\mathcal{Q}_{\text{tr}}$ and the aperiodic timer $\tau$, which resets at each transition and evolves as $\dot{\tau}\in[0,1/\tau_d]$ up to the jump threshold $\tau=1$. The switched-DAE model captures the dominant switching-induced transients and inter-mode behavior observed in the EMT simulations. Differences arise mainly from the DC power-flow linearization and neglected fast converter dynamics, particularly at GFL buses where the switched-DAE uses a frequency-divider approximation instead of explicit inner-loop and PLL dynamics. Overall, the results support the switched-DAE framework as a tractable reduced-order model for AI datacenter mode transitions.

\emph{b) Limiting behavior:}
To illustrate Theorem~\ref{thm: SGPAS} on the 39-bus system, we set $N_0=1$, imposing a standard dwell-time bound. Since $\xi\in\mathbb{R}^n$ is high-dimensional, we project onto $(\theta_{34},\Delta\omega_{34},\theta_{L,20})$, where $\theta_{L,20}$ distinguishes the four mode-dependent consistency spaces. Figure~\ref{fig:ols_39} shows the predicted set $\mathcal{Y}=\bigcup_q\mathcal{Y}_q$ (black curves) and the mode-dependent equilibria $\xi_q^*$ (stars). Their separation along $\theta_{L,20}$ reflects the different datacenter power levels $p_{dc}^q$. Each $\mathcal{Y}_q$ is generated by the mode-$q$ flow from a projected equilibrium $\Pi_q(\xi_{q'}^*)$. The high-dimensional consistency spaces are omitted because their low-dimensional projections provide no useful geometric information. Figure~\ref{fig:thm2_39} shows trajectories approaching $\mathcal{Y}$, with color denoting normalized time and dotted segments representing switching-induced jumps. After each jump onto the new consistency space, the state evolves under the active mode toward $\mathcal{Y}$. Thus, regardless of the initial condition, the long-run behavior is well approximated by $\mathcal{Y}$, consistent with Theorem~\ref{thm: SGPAS}.

\section{Stability under Inference Workloads: Stochastic Switching}
\label{sec: stochastic}
%
As described in Section \ref{sec:power}, during inference, the datacenter mode transitions are triggered by user requests whose arrival times and service durations are inherently random. Therefore, in this section, we extend the results of Section \ref{sec:formulation} by treating mode transitions during inference workloads as a stochastic phenomenon, and we study the regularity and stability properties of the power transmission dynamics \eqref{eq:sadae_ps} by modeling the overall system as a SHDS of the form \eqref{eqn:shy_sys_def}. Our model of datacenter inference will contain two sources of randomness: mode transitions and switching times.
\subsection{Stochastic Mode Transitions}
The operating mode of the data centers during inference is modeled as a random variable
$v_1 \in \mathcal{Q}_{\mathrm{in}}$. Importantly, $v_1$ characterizes the mode
entered by the \emph{aggregate} datacenter load at each transition, rather than
the mode associated with an individual inference request. For a \emph{single}
user request, mode transitions are deterministic and follow the cycle
$\texttt{id}\to\texttt{pre}\to\texttt{dec}\to\texttt{id}$. Randomness in
$v_1$ emerges only at the \emph{aggregate} level: when multiple user requests
are processed concurrently, the next mode observed by the power system depends
on the timing and completion of the individual requests. Consequently, the superposition of many asynchronous request-level transitions induces an
effectively random sequence of aggregate operating modes.

To model the stochastic mode transitions, let
\[
p_{q'}
:=
\mathbb{P}[v_1=q'],
\qquad q'\in\mathcal{Q}_{\mathrm{in}},
\]
where the probabilities $\{p_{q'}\}_{q'\in\mathcal{Q}_{\mathrm{in}}}$
are obtained from empirical transition frequencies of the aggregate
datacenter load and satisfy $p_{q'}\geq0$, and $\sum_{q'\in\mathcal{Q}_{\mathrm{in}}}p_{q'}=1$. Accordingly, at each mode transition, the next aggregate operating
mode is approximated as an independent categorical random variable,
\[
v_1\sim
\operatorname{Cat}
\bigl(\mathcal{Q}_{\mathrm{in}},\{p_{q'}\}_{q'\in\mathcal{Q}_{\mathrm{in}}}\bigr).
\]
This constitutes an aggregate stochastic approximation of the mode
sequence generated by many asynchronous inference requests. The
mode-dependent dwell time is subsequently sampled from the
conditional distribution $\mu_{2\mid v_1=q'}$. To support this model, we provide two arguments:
\vspace{0.1cm}
\par \noindent
a) Production data centers process many requests concurrently, each cycling asynchronously through inference modes with random dwell times. The aggregate power is therefore a superposition of many independent cycles. By the law of large numbers, the aggregate mode fractions concentrate around their stationary values, motivating the approximation of successive modes as i.i.d.\ draws from
$\operatorname{Cat}(\mathcal Q_{\mathrm{in}},\{p_{q'}\})$. Similar population-level approximations are used for thermostatically controlled loads (TCLs), where aggregate behavior is characterized by stationary mode fractions~\cite[Sections II,III]{radaideh2019}\cite[Section II]{kasis2022}.

b) Inference cycles occur over tens to hundreds of milliseconds, much faster than the seconds-scale electromechanical dynamics. Under this timescale separation, the slow power-system dynamics respond primarily to the stationary distribution of the fast inference process~\cite[Corollary 4.7]{yinZhang2013}. Thus, for an ergodic inference cycle with stationary distribution $\{p_{q'}\}$, we approximate successive aggregate modes as i.i.d.\ samples, motivating
$v_1\sim\operatorname{Cat}(\mathcal Q_{\mathrm{in}},\{p_{q'}\})$.

\vspace{0.1cm}
\par\noindent
\begin{rem}
\label{rem: breakdown}\normalfont{
This approximation requires sufficiently many concurrent requests and inference cycles per electromechanical time scale. At low arrival rates, a few requests may dominate the aggregate mode and mixing toward the stationary distribution may be slow~\cite{yang2024queueing}. Conversely, synchronized or bursty arrivals~\cite{burstgpt2024} may produce prolonged \texttt{pre} or \texttt{dec} phases, violating the independence or dwell-time assumptions. Such regimes, where the aggregate mode does not mix sufficiently fast relative to the power-system dynamics, are outside the scope of this analysis.}\qed
\end{rem}

\subsection{Stochastic Switching Times}
The stochastic switching time is parameterized by a random
timer-reset variable \(v_2\). The variable \(v_2\) is dimensionless
and determines the nominal minimum dwell time associated with the
next mode. More precisely, let \(T_{q'}^{\mathrm{nom}}\) denote the
nominal minimum dwell time associated with mode
\(q'\in\mathcal{Q}_{\mathrm{in}}\). We define
\begin{equation}\label{eq:normalized_dwell_time}
    v_2 := \frac{T_{q'}^{\mathrm{nom}}}{\tau_d},
    \qquad
    T_{q'}^{\mathrm{nom}}=\tau_d v_2,
\end{equation}
and assume that \(v_2\in[0,N_0]\). Accordingly, the
mode-dependent probability distributions introduced below describe
the normalized nominal dwell times
\(T_{q'}^{\mathrm{nom}}/\tau_d\), rather than the realized physical
dwell times. Under the timer inclusion introduced below, the realized
physical dwell time may be larger than \(T_{q'}^{\mathrm{nom}}\).

To distinguish physical-time parameters from the parameters of the
normalized timer-reset variable, define
\begin{equation}\label{eq:normalized_distribution_parameters}
    \bar{\lambda}_r := \tau_d\lambda_r,
    \qquad
    \bar{\beta}_{\texttt{pre}}
    := \tau_d\beta_{\texttt{pre}},
\end{equation}
where \(\lambda_r\) is the physical-time arrival rate and
\(\beta_{\texttt{pre}}\) is the physical-time rate parameter of a
Gamma distribution. Here, the Gamma distribution is parameterized
in terms of its shape and rate.

\vspace{0.1cm}
\emph{Prefill Mode}: It is well known that the physical
\emph{Prefill} duration \(T_{\texttt{pre}}\) scales approximately
linearly with the number of input tokens~\cite{wang2025systematic}. In
particular, $T_{\texttt{pre}}=
c_{\texttt{pre}}L_{\texttt{in}}$, where \(c_{\texttt{pre}}>0\) denotes the compute time per input token
and \(L_{\texttt{in}}\) is the prompt length. Since prompt lengths
depend on user requests, \(L_{\texttt{in}}\) is itself a random
variable, with empirical distributions exhibiting Zipf-like
behavior~\cite[Fig.~7]{burstgpt2024}. To capture variability in input prompt lengths, we model the
normalized nominal Prefill duration
$T_{\texttt{pre}}^{\mathrm{nom}}/\tau_d$ using a truncated
$\mathrm{Gamma}(\alpha_{\texttt{pre}},
\bar{\beta}_{\texttt{pre}})$ distribution supported on
$[0,N_0]$.

\vspace{0.1cm}
\emph{Idle Mode}: We assume that user requests arrive according to
a Poisson process with physical-time rate \(\lambda_r\).
Consequently, the physical inter-arrival time is exponentially
distributed with rate \(\lambda_r\). Since the normalized
timer-reset variable is obtained by dividing the nominal
inter-arrival time by \(\tau_d\), we model it using an
\(\mathrm{Exp}(\bar{\lambda}_r)\) distribution truncated to
\([0,N_0]\), where
\(\bar{\lambda}_r=\tau_d\lambda_r\).

\emph{Decode Mode}: Finally, the physical \emph{Decode} duration
satisfies $T_{\texttt{dec}}=
c_{\texttt{dec}}L_{\texttt{out}}$, where \(L_{\texttt{out}}\) is the number of generated output tokens.
Output lengths are heavy-tailed~\cite{wang2026robust,zheng2026scheduling}
and are empirically well fitted by a log-\(t\)
distribution~\cite[Fig.~1]{zheng2026scheduling}. Since LLM inference services can impose an upper bound
$N_{\texttt{cap}}$ on the number of generated output
tokens~\cite{yang2024queueing}, the corresponding maximum
nominal Decode duration is $T_{\max}=c_{\texttt{dec}}N_{\texttt{cap}}$.
Therefore, the normalized timer-reset variable is supported
on $[0,N_0]$., where
\begin{equation}\label{eq:normalized_timer_bound}
    N_0
    :=
    \frac{T_{\max}}{\tau_d}
    =
    \frac{c_{\texttt{dec}}N_{\texttt{cap}}}{\tau_d}.
\end{equation}

\vspace{0.2cm}
    Throughout all three modes, we adopt the same normalized upper bound
\(N_0\), since the Decode mode typically dominates the per-request
duration~\cite{distserve2024}. The conditional distributions of the
normalized timer-reset variable are therefore modeled as
\begin{equation}\label{eq: dwell_distributions}
\resizebox{0.92\columnwidth}{!}{$\displaystyle
v_2\mid v_1=q'
\sim
\begin{cases}
\operatorname{TrExp}(\bar{\lambda}_r;[0,N_0]),
    & q'=\texttt{id},\\[3pt]
\operatorname{TrGamma}
(\alpha_{\texttt{pre}},\bar{\beta}_{\texttt{pre}};[0,N_0]),
    & q'=\texttt{pre},\\[3pt]
\operatorname{HT}_{\mathrm{norm}}([0,N_0]),
    & q'=\texttt{dec}.
\end{cases}$}
\end{equation}
Here, $\operatorname{TrExp}$ and $\operatorname{TrGamma}$ denote,
respectively, the exponential and Gamma distributions conditioned
on the interval $[0,N_0]$ (their normalized densities are given in
Appendix~\ref{appen:integrals}). Moreover,
$\operatorname{HT}_{\mathrm{norm}}([0,N_0])$ denotes a general
heavy-tailed distribution for the normalized nominal Decode duration
on $[0,N_0]$, obtained by scaling the corresponding physical-time
distribution by $1/\tau_d$. Possible choices include a truncated
log-$t$ distribution~\cite[Fig.~1]{zheng2026scheduling} or a truncated
power law~\cite{wang2026robust}.   
\subsection{Global Recurrence Under Stochastic Switching}
\label{subsec: shds}
We are now ready to cast the overall power transmission system with co-located datacenter loads subject to an inference workload as a stochastic hybrid system of the form \eqref{eqn:shy_sys_def}. To that end, let $\{(v_{1,k},v_{2,k})\}_{k=1}^{\infty}$ be an
i.i.d.\ sequence of random vectors drawn from the joint probability
measure
\begin{equation}
    \mu(dv_1,dv_2)
    =
    \mu_1(dv_1)\mu_{2\mid v_1}(dv_2),
\end{equation}
where $\mu_1 =
\operatorname{Cat}(\mathcal{Q}_{\mathrm{in}},\{p_{q'}\})$ and
$\mu_{2\mid v_1=q'}$ is given by~\eqref{eq: dwell_distributions}.
We emphasize that $\mu$ is generally not a product measure, since the
conditional distribution $\mu_{2\mid v_1}$ depends on $v_1$. Thus,
although the random vectors
$\{(v_{1,k},v_{2,k})\}_{k=1}^{\infty}$ are independent across $k$,
their components $v_{1,k}$ and $v_{2,k}$ are generally statistically
dependent.

\vspace{0.2cm}
\par \noindent
Using the change of variables $\tilde\xi = \xi - \xi_q^*$ introduced in
Section~\ref{sec:formulation}, and writing $y = (\tilde\xi, q, \tau)$, the overall dynamics of the system are given by the following SHDS:
\begin{subequations}\label{eq: shds}
\begin{align}
\dot{\tilde\xi}=&\Pi_q^{\mathrm{diff}} A_q\tilde\xi,~~~\dot{q}=0,~~\dot{\tau}\in\left[-\frac{1}{\tau_d}, 0\right]~~~
    y \in C,\\[2mm]
    \begin{bmatrix}\tilde\xi^+\\ q^+\\ \tau^+\end{bmatrix}
    &=
    \begin{bmatrix}\textcolor{black}{L_{v_1}\tilde\xi + \Delta_{q,v_1}}\\
    v_1\\ v_2\end{bmatrix},
    \quad y \in D, \\[0.2cm]
    C &= \bigcup_{q\in\mathcal{Q}_{\text{in}}}\text{span}(V_q)\times\{q\}\times[0,N_0], \\[0.2cm] 
    D &= \bigcup_{q\in\mathcal{Q}_{\text{in}}}\text{span}(V_q)\times\{q\} \times\{0\},   
\end{align}
\end{subequations}
where $\Delta_{q,v_1} \coloneqq \Pi_{v_1}(\xi_q^*) - \xi_{v_1}^*$. Note that this system has the form of \eqref{eqn:shy_sys_def} with $v_k = (v_{1,k},v_{2,k}) \sim \mu$ and $\mu(dv_1,dv_2)
=
\mu_1(dv_1)\mu_{2\mid v_1}(dv_2)$. Under the timer dynamics $\dot{\tau}\in\left[-\frac{1}{\tau_d},0\right]$ and $\tau^+=v_2$, the quantity \(\tau_dv_2=T_{q'}^{\mathrm{nom}}\) represents the
minimum time required for the timer to reach the jump condition
\(\tau=0\). This minimum is attained when
\(\dot\tau=-1/\tau_d\). Values of
\(\dot\tau\in(-1/\tau_d,0]\) allow longer inter-transition times and
therefore capture uncertainty or additional delays in the realized
physical dwell time. Consequently, we have that $T_{\mathrm{dwell}}\geq \tau_d v_2$.

Before stating the main stability results for \eqref{eq: shds}, we introduce two assumptions extending those in~\cite{povedaAbdelgalil2025} beyond the common-equilibrium setting. 
\begin{assump}
\label{assump: flow_LMI}\normalfont{
There exist $\mu>0$, $\theta>0$ and $\kappa^\star\in(0,1)$ such that,
for each $q\in\mathcal{Q}_{\mathrm{in}}$, there exists a pair of
symmetric matrices $(P_q,\tilde{N}_q)$ that are positive definite on
$\operatorname{span}(V_q)$ and satisfy the following inequalities on
$\operatorname{span}(V_q)$:
\begin{align}
A_q^\top \Pi_q^{\mathrm{diff},\top}P_q
+P_q\Pi_q^{\mathrm{diff}}A_q
+\frac{\mu}{\tau_d \cdot N_0} P_q
&\preceq -\tilde{N}_q,
\\
(1+\theta)
\sum_{q'\in\mathcal{Q}_{\mathrm{in}}}
p_{q'}\widetilde{I}_{q'}(\mu)
L_{q'}^\top P_{q'}L_{q'}
&\preceq
\kappa^\star P_q.
\end{align}}
\end{assump}

\vspace{0.1cm}
\par \noindent
Analogous to the quantity $\alpha$  in Theorem \ref{thm:ISS_like}, we define the probability-weighted residual:
\begin{equation}
\label{eq: alpha_gamma}
    \alpha_\Gamma(\mu) \;:=\; \max_{q\in\mathcal{Q}_{\text{in}}}\sum_{q'\in\mathcal{Q}_{\text{in}}}
    p_{q'}\,\tilde{I}_{q'}(\mu)\,
    \textcolor{black}{\Delta_{q,q'}^\top P_{q'} \Delta_{q,q'}}.
\end{equation}
Compared to the worst-case $\alpha$ in Theorem~\ref{thm:ISS_like},
$\alpha_\Gamma(\mu)$ is probability-weighted. That is, modes visited rarely (small $p_{q'}$) and modes associated
with larger timer-reset values (small $\widetilde{I}_{q'}(\mu)$)
contribute less to $\alpha_\Gamma$.  Based on this, the following recurrence theorem is the main result of this section, and it leverages the construction of the set
\begin{equation}
\mathcal{A}_{\text{in}}=\{0\}\times\mathcal{Q}_{in}\times[0,N_0].
\end{equation}
The proof is presented in the Appendix.
\begin{thm}\label{thm: shds_recurrence}
Suppose that Assumptions~\ref{assumpregdae},
\ref{assump: regularity}, and
\ref{assump: flow_LMI} hold with respect to
$\mathcal{Q}_{\mathrm{in}}$. Then, for every $\varepsilon>0$,
the open (relative to $C\cup D$) bounded set
\begin{equation}\label{eq: recurrent_set}
    \mathcal{O}_\varepsilon \;:=\; \Bigl\{x \in C\cup D :
    |x|_{\mathcal{A}_{\text{in}}} < r\!\left(\alpha_\Gamma(\mu)\right) + \varepsilon\Bigr\},
\end{equation}
is uniformly globally recurrent (UGR) for the
SHDS~\eqref{eq: shds}, with $r(\alpha_\Gamma(\mu))=\sqrt{(c_1/c_2)\alpha_\Gamma(\mu)}$, $c_1=(1+1/\theta)$, and $c_2=e^{-\mu}\underline{p}(1-\kappa^{\star})$. \hfill $\square$ 
%
%
\end{thm}
\begin{rem}\normalfont{
Theorem~\ref{thm: shds_recurrence} establishes UGR to a neighborhood of $\mathcal{A}_{\mathrm{in}}$ with radius $r(\alpha_\Gamma(\mu))$, determined by the probability-weighted residual~\eqref{eq: alpha_gamma}. Thus, from any bounded set of initial conditions, the state reaches this neighborhood with probability approaching one. If all equilibria coincide, $\alpha_\Gamma(\mu)=0$ and the Lyapunov--Foster conditions yield UGASp and recovering stochastic asymptotic stability of the set $\mathcal{A}_{\mathrm{in}}$.} \hfill \qed
\end{rem}
The following stylized example illustrates the result of Theorem \ref{thm: shds_recurrence}.
\begin{ex}
\label{ex:stochastic}
To illustrate Theorem~\ref{thm: shds_recurrence}, we retain the
same switched DAE as in Example~\ref{ex: example_7_nanez}. The
system matrices $E_q$, $A_q$, $b_q$, and the Lyapunov matrices
$P_q$, $O_q$ are unchanged, but we replace the deterministic
switching mechanism with stochastic i.i.d.\ normalized timer-reset
values drawn from a uniform distribution. 
Specifically, for each mode transition we use
\begin{equation*}
    v_2 \mid v_1 = q' \;\sim\; \mathrm{Uniform}(0.1,\, N_0),
    \quad q' \in \{1,2\},
\end{equation*}
with $N_0 = 2$, $\mu = 6.5$, $\theta = 0.35$,
$\tau_d = 4$, and equal mode probabilities $p_1 = p_2 = 0.5$.
\vspace{0.1cm}
\par \noindent
The flow condition
requires $\mu\cdot \sfrac{1}{\tau_d}/N_0 < \underline{m}/\overline{p}$, i.e.\
$6.5\times 0.25/2 = 0.8125 < 1$, giving a net flow decrease rate
$\rho = 1 - 0.8125 = 0.1875 > 0$. Assumption~\ref{assump: flow_LMI} is satisfied with
\begin{equation*}
    \kappa = (1+\theta)\tilde{I}_q(\mu)
    = (1.35)(0.117) \approx 0.1579 < 1,
\end{equation*}
thereby verifying both inequalities in
Assumption~\ref{assump: flow_LMI}.
The probability-weighted residual~\eqref{eq: alpha_gamma} is $\alpha_\Gamma(\mu)
    = \tfrac{1}{2}\,\tilde{I}_q(\mu)\cdot 9
    \approx 0.5265$, where the factor $9=
\left\|
\Pi_{q'}(\xi_q^\star)-\xi_{q'}^\star
\right\|_{P_{q'}}^2$ is the same projected-equilibrium residual as in
Example~\ref{ex: example_7_nanez}, multiplied here by the
weight $\widetilde I_q(\mu)\approx0.117$, instead of the
worst-case factor $1$ appearing in $\alpha$. 
Figure~\ref{fig:example_stochastic} illustrates these results.
Panel~A shows the distance $|x|_{\mathcal{A}_{\text{in}}}$ remaining bounded and
repeatedly returning to a recurrent neighborhood, consistent with uniform global recurrence.
Panel~B shows the state trajectories $x_1, x_2$ oscillating under
the stochastic mode sequence.
Panel~C shows the i.i.d.\ normalized timer-reset values drawn from
$\mathrm{Uniform}[0.1,2]$. Since $\tau_d=4$, these values correspond
to nominal minimum physical dwell times in $[0.4,8]$.
\hfill \qed
\end{ex}
\begin{figure}[t]
    \centering
    \includegraphics[width=0.99\linewidth]{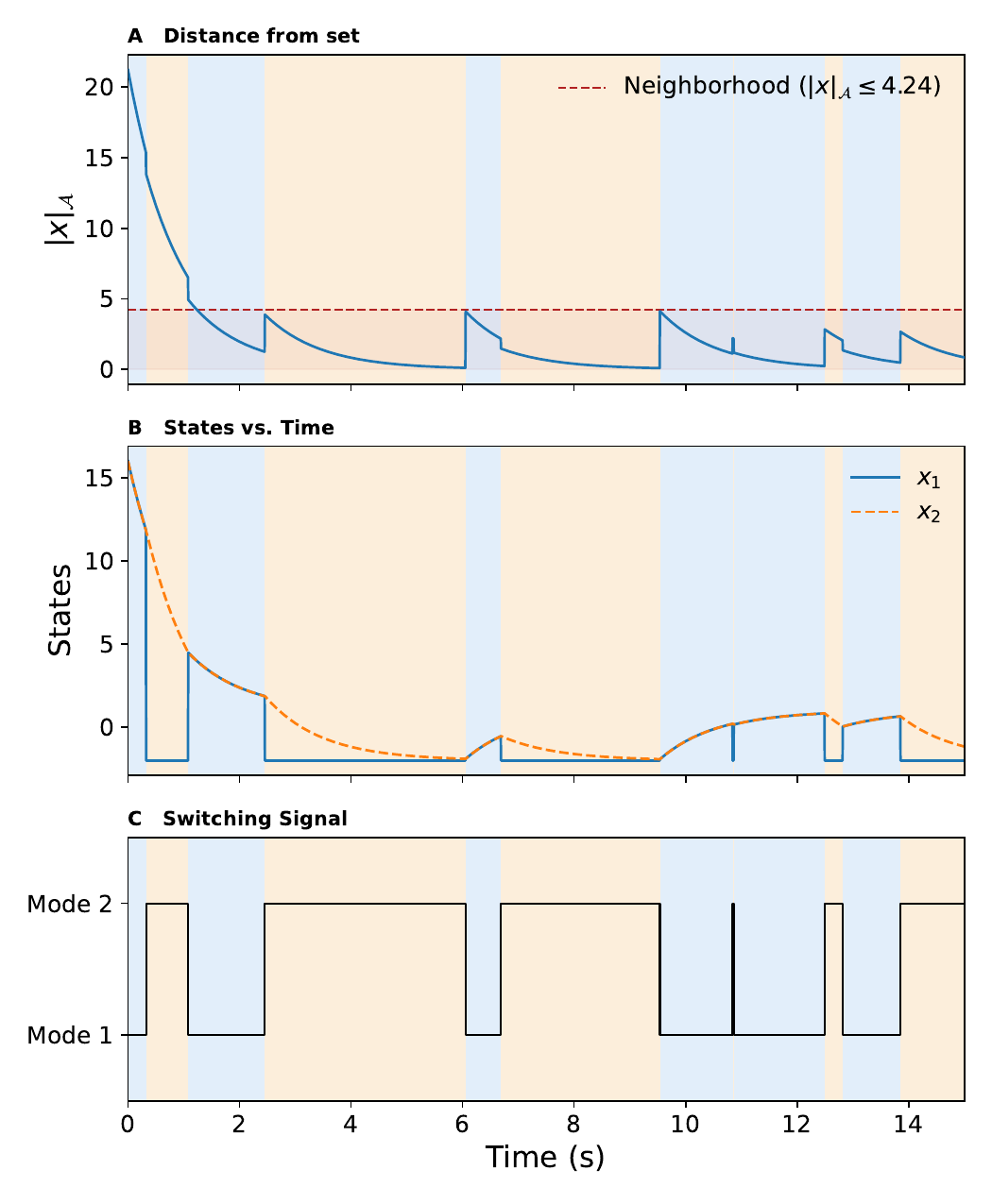}
    \caption{Stochastic two-mode example illustrating Theorem~\ref{thm: shds_recurrence}: (A) distance $|x|{\mathcal A{\mathrm{in}}}$, (B) state trajectories, and (C) switching signal with uniformly distributed timer resets.}
    \label{fig:example_stochastic}
\end{figure}
%
%

\begin{figure*}[t]
    \centering

    \begin{subfigure}[t]{0.49\textwidth}
        \centering
        \includegraphics[width=\textwidth]{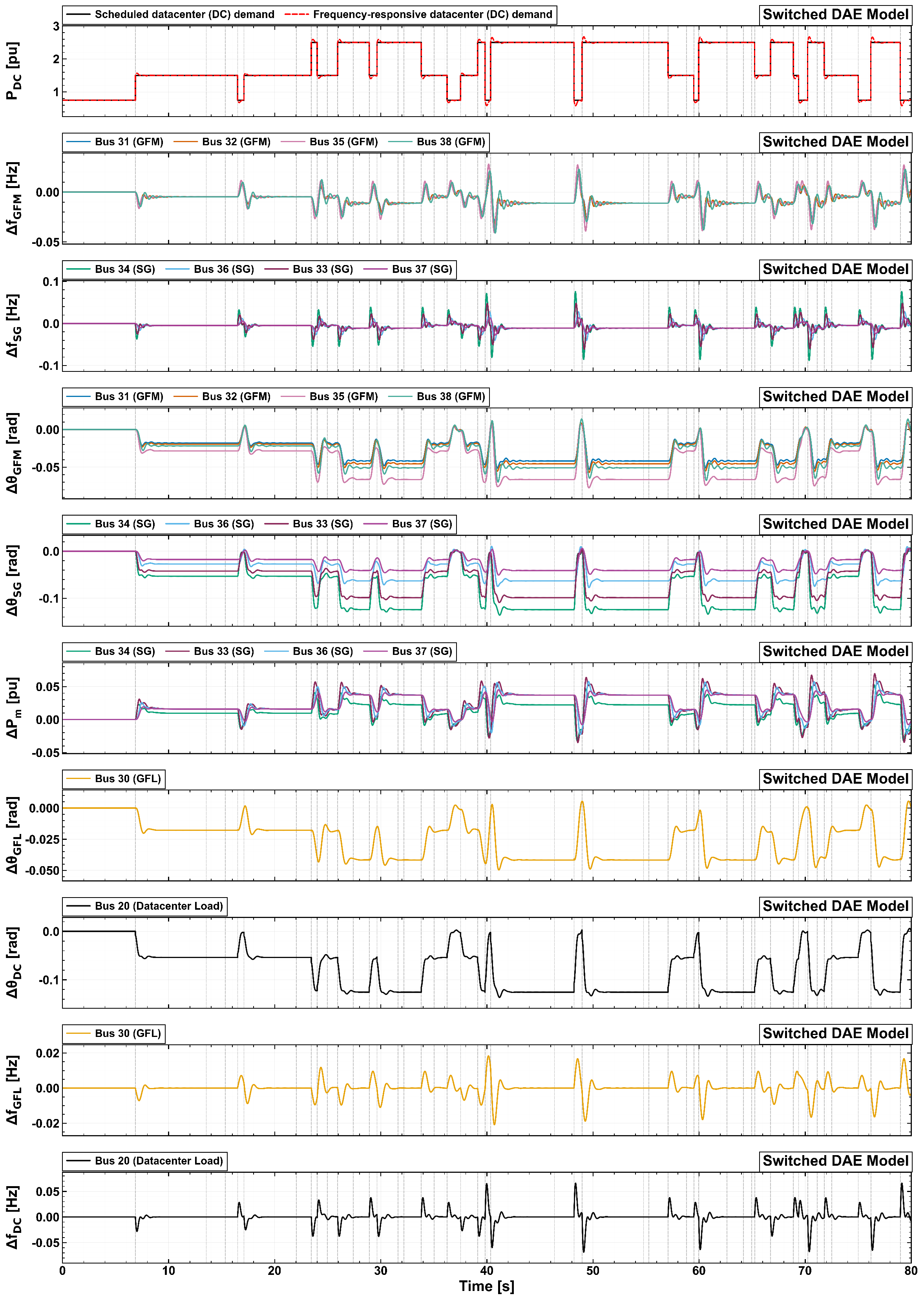}
        \caption{Stochastic switched DAE model}
        \label{fig:stochastic_dae}
    \end{subfigure}
    \hfill
    \begin{subfigure}[t]{0.49\textwidth}
        \centering
        \includegraphics[width=\textwidth]{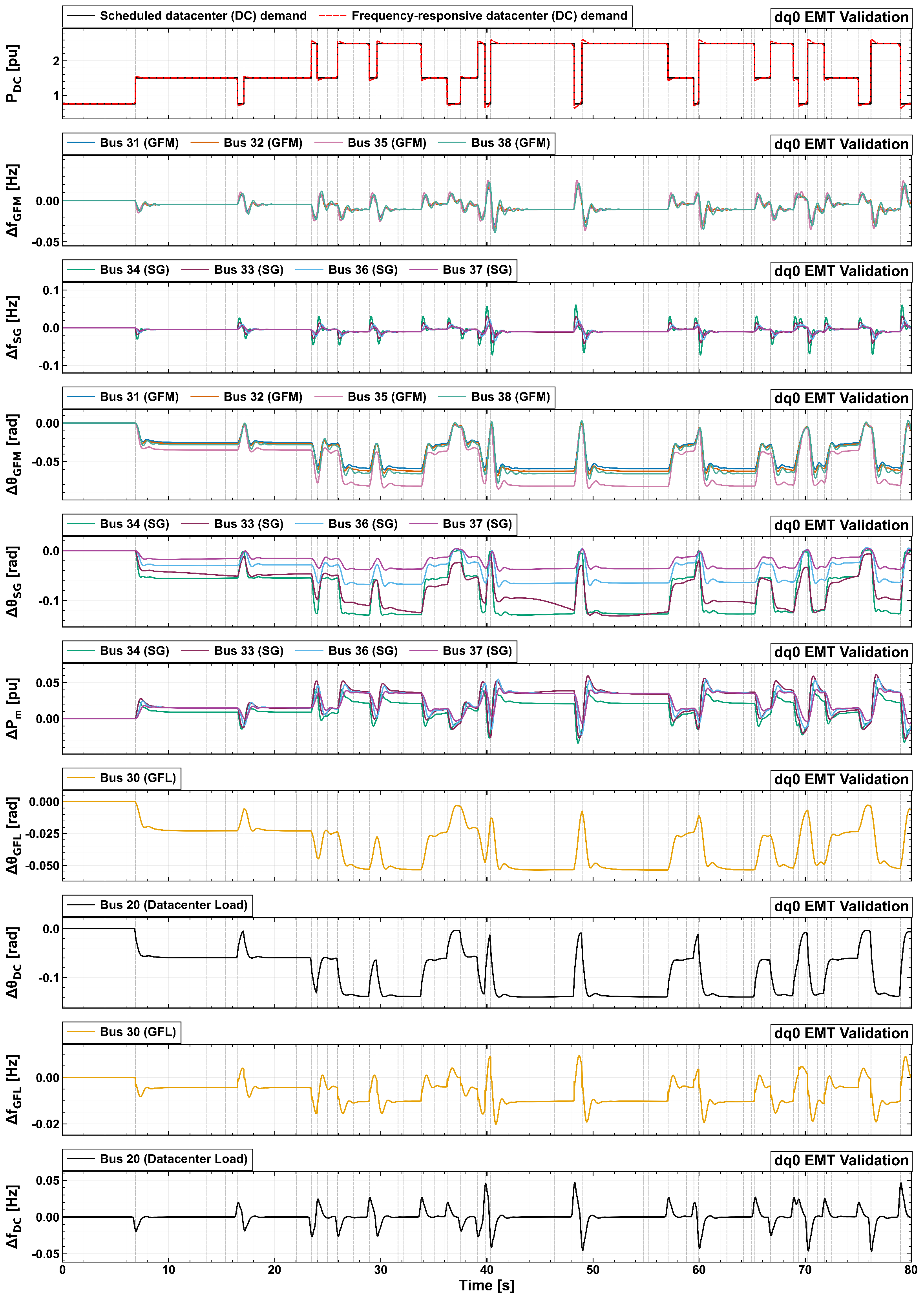}
        \caption{dq0 EMT validation}
        \label{fig:stochastic_emt}
    \end{subfigure}

    \vspace{-0.3em}
    \caption{Stochastic inference-phase switching on the modified IEEE 39-bus system with AI datacenter loads: (a) switched-DAE model and (b) high-fidelity dq0 EMT validation in PSID.jl, both under switching over $\mathcal{Q}_{\text{in}}=\{id,pre,dec\}$.}
    \label{fig:stochastic_switched_dae_emt_validation_responses}
    \vspace{-0.4cm}
\end{figure*}
\subsection{Application to AI Power Transmission Systems with Inference Workloads: Experiments and EMT Validations}
\label{subsec:stochastic_inference}

The stochastic scenario drives the 39-bus system with an inference-phase
switching signal over $\mathcal{Q}_{\text{in}} = \{id, pre, dec\}$, with mode
transitions drawn i.i.d.\ from $\mathrm{Cat}(\mathcal{Q}_{\text{in}}, \{p_{q'}\})$
and dwell times drawn from the per-mode distributions
of~\eqref{eq: dwell_distributions}: truncated exponential idle, truncated
Gamma prefill, and truncated heavy-tailed decode. The datacenter at Bus~20
cycles through power levels $p_{dc}^{pre} > p_{dc}^{dec} > p_{dc}^{id}$,
producing irregular switching intervals whose heavy-tailed decode dwells
are visible as the longer orange phases in
Figure~\ref{fig:stochastic_switching_and_lyapunov}.

\par Figure~\ref{fig:stochastic_switched_dae_emt_validation_responses}
compares the stochastic switched-DAE model (Panel a) against the dq0 EMT
simulation (Panel b). The stochastic mode sequence produces aperiodic
switching, yet both models show bounded frequency and angle deviations
across all generator, inverter, and datacenter buses, qualitatively
consistent with the uniform global recurrence guaranteed by
Theorem~\ref{thm: shds_recurrence}.

\par
Taken together, Figures \ref{fig:stochastic_switched_dae_emt_validation_responses} and \ref{fig:stochastic_switching_and_lyapunov} indicate that the switched-DAE model remains a
faithful reduced-order representation of the transmission system even when mode
transitions are driven by the random, request-level dynamics of AI inference
workloads rather than a fixed computational schedule. The aperiodic, heavy-tailed
dwell times of the stochastic mode sequence produce irregular excursions in
frequency and angle across the generator, inverter, and datacenter buses, yet both
the analytical model and the dq0 EMT simulation confine these excursions to a
comparable bounded range and return them to a neighborhood of the mode equilibria
after each switch, in agreement with the recurrence behavior certified by
Theorem \ref{thm: shds_recurrence}. This agreement persists despite the EMT model's additional fidelity, in
particular its explicit representation of converter switching dynamics, PLL
synchronization, and transmission-line electromagnetic transients that are
abstracted away in the reduced-order representation, suggesting that the timescale
separation underlying the stochastic switched-DAE model is not merely a modeling
convenience but a reasonable approximation of the aggregate power-system response
to inference-driven datacenter load switching.

\begin{figure*}[t]
    \centering
        \includegraphics[width=\textwidth]{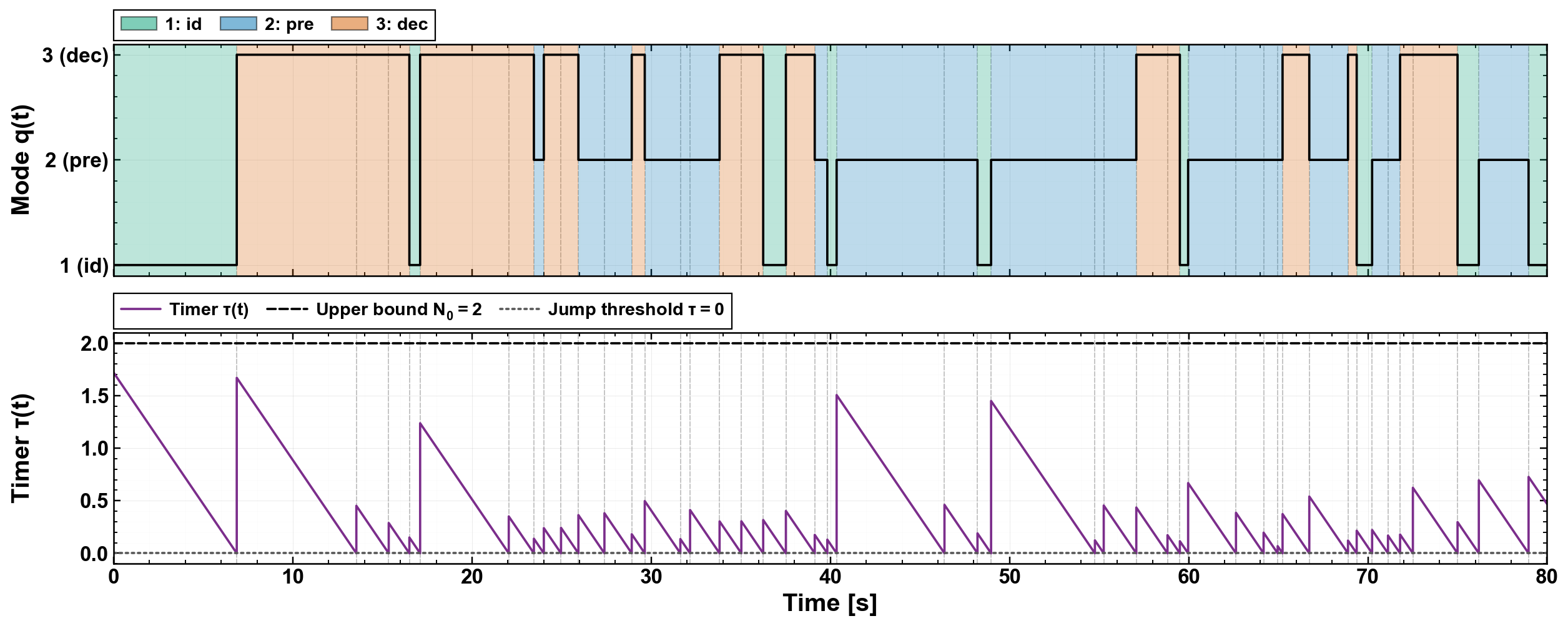}
   \label{fig:stochastic_lyapunov_certificate}
    \vspace{-0.3em}
    \caption{Stochastic switching signal and the timer for the scenario in Fig.~\ref{fig:stochastic_switched_dae_emt_validation_responses}. The figure shows the stochastic mode sequence and countdown timer $\tau(t)$. }
\label{fig:stochastic_switching_and_lyapunov}
\vspace{-0.4cm}
\end{figure*}
\section{Conclusion}
\label{sec:conclusion}
This paper studied switched affine DAEs with mode-dependent equilibria using hybrid systems tools. Leveraging the quasi-Weierstrass decomposition, we derived a hybrid representation capturing flows on consistency spaces and switching-induced resets. We established Lyapunov conditions for practical stability with an explicit $\mathcal{KL}$ bound (Theorem~\ref{thm:ISS_like}) and characterized the limiting behavior through a computable $\omega$-limit set that is semi-globally practically stable under slow switching (Proposition~\ref{prop: ols_dae}, Theorem~\ref{thm: SGPAS}). These results extend switched DAE stability theory beyond the common-equilibrium setting. For stochastic switching motivated by AI inference workloads, we established uniform global recurrence to a neighborhood of the mode equilibria using Foster--Lyapunov techniques and mode-dependent dwell-time distributions (Theorem~\ref{thm: shds_recurrence}). The neighborhood depends on a probability-weighted inter-equilibrium residual, and the result recovers uniform global asymptotic stability in probability when the projected mode equilibria coincide. We applied the framework to transmission networks with co-located AI data centers, where workload transitions induce distinct consistency spaces and state resets. Simulations on the IEEE 39-bus system for deterministic training and stochastic inference workloads were consistent with high-fidelity dq0 electromagnetic transient simulations, supporting the switched-DAE model as a tractable representation of datacenter workload switching. Future work will consider correlated and Markov-driven mode transitions, operating regimes where the i.i.d.\ approximation breaks down, and distribution networks with co-located data centers.
\appendices

\section{Auxiliary Definitions and Lemmas}\label{Appendix:auxdef}
In this appendix, we collect some auxiliary definitions and results on hybrid systems and switched affine DAEs. We begin with some standard definitions taken from \cite[Sec.~6.3]{goebel2012hybrid}.
\begin{defn}\label{definitionReachable}
    Let $K \subset \mathbb{R}^n$ be a compact set. The (infinite horizon) reachable set from $K$ for the hybrid inclusion \eqref{eqn:hy_sys_def}, denoted as $R(K)$, is defined as $R(K) \coloneqq \{z \in \mathbb{R}^n  :  z=\xi(t,j), \xi \in \mathcal{S}(K), (t,j) \in \text{dom}(\xi)\}$,  where $\mathcal{S}(K)$ denotes the set of solutions to  \eqref{eqn:hy_sys_def} with initial conditions in $K$.
\hfill \qed
\end{defn}
\begin{defn}\label{defn:omega_set}
    Given a set $K$, the $\Omega$-limit set from $K$, denoted by $\Omega(K)$ is defined as  $\Omega(K) \coloneqq \{z \in \mathbb{R}^n: z = \lim_{i \to \infty} \xi_i(t_i,j_i),~\xi_i \in \mathcal{S}(K), (t_i,j_i) \in \text{dom}(\xi_i),~ \lim_{i \to \infty} t_i + j_i = \infty\}$, where $\mathcal{S}(K)$ denotes the set of solutions to system \eqref{eqn:hy_sys_def} with initial conditions in $K$. \hfill \qed
\end{defn} 

%
\noindent 
The following definition will be instrumental for the study of hybrid systems \eqref{eqn:hy_sys_def} that depend on small parameters $\epsilon>0$.
%
%
We first review auxiliary results for switched affine DAEs over a general mode set $\mathcal{Q}\coloneqq\{1,\dots,l\}$, applicable to both training and inference workloads. We begin with the consistency spaces~\cite{nanez2017invariance}, which characterize states compatible with the algebraic constraints of each mode:
\begin{defn}
    \label{defn: cons_space}
Given $q \in \mathcal{Q}$, the \emph{consistency space} $\mathcal{C}_q$ of the $q^{\text{th}}$ mode of system \eqref{eqn: affine_DAE} is defined as:
    \begin{align*}
        \mathcal{C}_q \coloneqq \{ &\xi_0 \in \mathbb{R}^n \, : \, \exists~\tau>0~\text{and}~\text{a continuously differentiable} \\
        & \text{ function } \xi : [0, \tau] \to \mathbb{R}^n \text{ s.t. }\, E_q \Dot{\xi}(t) = A_q \xi(t) + b_q,\\
        &~~\forall \, t \in (0,\tau)\}.
    \end{align*}
    In other words, $\mathcal{C}_q$ denotes the set of initial conditions for which there exists a solution to \eqref{eqn: affine_DAE}.
    \hfill \qed
\end{defn}
The following change of basis allows us to decouple system \eqref{eqn: affine_DAE} into its differential and algebraic parts, called the Quasi-Weierstrass form \cite[Theorem 2.6]{berger2012quasi}: 
\begin{lem}
    \label{lem: QWF}
Suppose that the pair $(E_q, A_q)$ is regular, where $q \in \mathcal{Q}$. Then, there exist invertible matrices $T_q \coloneqq [V_q, W_q]$ and $S_q \coloneqq [E_qV_q, A_qW_q]^{-1}$ that transform $(E_q, A_q)$ into the \emph{Quasi-Weierstrass form}:
    \begin{equation}\label{QWform}
        (S_q E_q T_q, S_q A_q T_q) = \left ( \begin{bmatrix}
            I_{n^q_1} & 0 \\
            0 & N_q
        \end{bmatrix}, \begin{bmatrix}
            J_q & 0 \\
            0 & I_{n^q_2}
        \end{bmatrix} \right ),
    \end{equation}
    for some $J_q \in \mathbb{R}^{n_1^q \times n^q_1}$ and a nilpotent matrix $N_q \in \mathbb{R}^{n_2^q \times n_2^q}$, where $(N_q)^{n_2^q} = 0$, $n_1^q + n_2^q = n$, $V_q \in \mathbb{R}^{n \times n_1^q}$, and $W_q \in \mathbb{R}^{n \times n^q_2}$. \hfill \qed
\end{lem}
Using the Quasi-Weierstrass form \eqref{QWform}, for each $q \in \mathcal{Q}$, we introduce the  variables $(v,w)= T_q^{-1}\xi$, to obtain 
    \begin{align*}
        E_q T_q \begin{bmatrix}
            \Dot{v} \\ \Dot{w}
        \end{bmatrix} &= A_q T_q \begin{bmatrix}
            v \\ w
        \end{bmatrix} + b_q \\
        S_q E_q T_q \begin{bmatrix}
            \Dot{v} \\ \Dot{w}
        \end{bmatrix} &= S_q A_q T_q \begin{bmatrix}
            v \\ w
        \end{bmatrix} + S_q b_q \\
        \begin{bmatrix}
            I_{n^q_1} & 0 \\
            0 & N_q
        \end{bmatrix} \begin{bmatrix}
            \Dot{v} \\ \Dot{w}
        \end{bmatrix}&= \begin{bmatrix}
            J_q & 0 \\
            0 & I_{n^q_2}
        \end{bmatrix}\begin{bmatrix}
                v \\ w
            \end{bmatrix} + S_q b_q,
    \end{align*}
    leading to the decoupled equations: $\Dot{v} = J_qv + b^d_q$, and $N_q\Dot{w} = w + b^a_q$ where $S_qb_q =(b^d_q,b^a_q)$. The following lemma is an extension of \cite[Definition 3]{nanez2017invariance}.
\begin{lem}
\label{lem: characterization}
Let Assumption \ref{assump: regularity} hold. Then, $\mathcal{C}_q =  \{\xi \in \mathbb{R}^n \, : \, \xi \in (\text{span}(V_q) - W_q b^a_q)\}$,~$\forall~q\in\mathcal{Q}$. 
\hfill \qed
\end{lem}
\noindent
\textbf{Proof:} Consider the change of variables $(v,w) = T_q^{-1}\xi$ from
Definition~\ref{lem: QWF}. Under this change of coordinates, System \eqref{eqn: affine_DAE}
decouples as $\dot v = J_qv + b_q^d$, $N_q\dot w = w + b^a_q$. The first equation is a linear ordinary differential equation and therefore
admits a $C^1$ solution from any initial condition $v(0)\in\mathbb{R}^{n^q_1}$. For the second equation, for each fixed $q\in\mathcal{Q}$ we have that $b^a_q$ is constant, thus the
constant function $w(t)\equiv -b_q^a$ satisfies $\dot w = 0$ and
$N_q\dot w = 0 = w + b_q^q$ for every nilpotent $N_q$, so it is a
solution. We now show it is the \emph{only} admissible one: subtracting
this particular solution from any other, the difference
$\tilde w := w - (-b_q^a)$ must satisfy the homogeneous equation
$N_q\dot{\tilde w} = \tilde w$. Indeed, since $N_q$ is nilpotent, let $k \geq 1$ be smallest such that
$N_q^k = 0$. We show $N_q^i \tilde w \equiv 0$ for
$i = k-1, k-2, \ldots, 0$ by downward induction on $i$, using at each
step only the single derivative $\dot{\tilde w}$. First, we multiply $\tilde w = N_q \dot{\tilde
w}$ on the left by $N_q^{k-1}$ to get $N_q^{k-1} \tilde w = N_q^{k-1} \big(N_q \dot{\tilde w}\big)
= N_q^{k} \dot{\tilde w} = 0$, using $N_q^k = 0$. Hence $N_q^{k-1} \tilde w \equiv 0$. Now, Suppose $N_q^{i+1} \tilde w \equiv 0$ for some
$0 \leq i \leq k-2$. Since this is the identically zero function, it
is (trivially) differentiable, and $N_q^{i+1} \dot{\tilde w} = \frac{d}{dt}\big(N_q^{i+1} \tilde w\big)
\equiv 0$. Multiplying $\tilde w = N_q \dot{\tilde w}$ on the left by $N_q^i$
then gives $N_q^{i} \tilde w = N_q^{i} \big(N_q \dot{\tilde w}\big)
= N_q^{i+1} \dot{\tilde w} = 0$, so $N_q^{i} \tilde w \equiv 0$. Downward induction from $i = k-1$ to $i = 0$ yields $N_q^0 \tilde w = \tilde w
\equiv 0$, so $\tilde w \equiv 0$ is the unique $C^1$ solution of the
homogeneous equation. Hence $w(0) = -b^a_q$ is the only admissible
initial condition for the $w$-component. Therefore, a classical $C^1$ solution of $E_q\dot\xi = A_q\xi + b_q$
exists starting from $\xi_0\in\mathbb{R}^n$ if and only if the
$w$-component of $T_q^{-1}\xi_0$ equals $-b_q^a$, with the
$v$-component free. Equivalently,
\begin{align*}
    \mathcal{C}_q
    &= \left\{\, \begin{bmatrix} V_q & W_q \end{bmatrix}
        \begin{bmatrix} v_0 \\ -b_q^a \end{bmatrix}
        \;:\; v_0 \in \mathbb{R}^{n^q_1} \,\right\} \\
    &= \left\{\, V_qv_0 - W_qb_q^a \;:\; v_0\in\mathbb{R}^{n^q_1} \,\right\}
    = \mathrm{span}(V_q) - W_qb_q^a,
\end{align*}
which completes the proof. \hfill $\blacksquare$

\vspace{0.1cm}
To ensure that solutions of system \eqref{eqn: affine_DAE} are well defined for all time, the state must remain consistent whenever the DAE \eqref{eqn: affine_DAE} switches modes. To do so, we use the notion of consistency projectors \cite[Definition 4]{nanez2017invariance}: 
\begin{defn}
    \label{defn: CP}
Given a mode $q \in \mathcal{Q}$, the \emph{consistency projector} $\Pi_q : \mathbb{R}^n \to \mathcal{C}_q$ of \eqref{eqn: affine_DAE} is given by $\Pi_q(\xi):=L_q \xi + d_q$, where $L_q \;:=\; T_q \begin{bmatrix} I_{n^q_1} & 0 \\ 0 & 0 \end{bmatrix} T^{-1}_q,$ and $d_q \;:=\; -W_q b^a_q$, and where $T_q$ and $W_q$ are given in Lemma \ref{lem: QWF}. \hfill \qed
\end{defn}

Standard Caratheodory solutions to each subsystem of \eqref{eqn: affine_DAE} can be obtained as solutions of an associated ODE. The following notion of differential projector is instrumental to isolate the ODE-governed component of the DAE dynamics:
\begin{defn}
    \label{defn: diff_proj}
  Given a $q \in \mathcal{Q}$, the \emph{differential projector} of the $q^{th}$ mode of \eqref{eqn: affine_DAE} is given by $\Pi_{q}^{\text{diff}}  = T_q \begin{bmatrix}
            I_{n^q_1} & 0 \\
            0 & 0_{n^q_2}
        \end{bmatrix} S_q$, where $T_q$ and $S_q$ are given by Lemma \ref{lem: QWF}. \hfill $\square$ 
\end{defn}
To extract the non-ODE component of the DAE dynamics, we introduce the impulse projector, which also enters the DAE solution formula.
\begin{defn}
    \label{defn: imp_proj}
   Given $q \in \mathcal{Q}$, the \emph{impulse projector} of the $q^{\text{th}}$ mode of the system \eqref{eqn: affine_DAE} is given by $\Pi_q^{\text{imp}} = T_q \begin{bmatrix}
            0 & 0 \\ 0 & I_{n^q_2} 
        \end{bmatrix} S_q$, where $T_q$ and $S_q$ are given by Lemma \ref{lem: QWF}. \hfill $\square$ 
\end{defn}

\section{Proofs for the Training-Phase Switching Case}
\label{appen:proofs}
\subsection{Proof of Theorem \ref{thm:ISS_like}}
To aid in our analysis, we consider the Lyapunov candidate $V(x) = {\rm e}^{\mu \tau} \Tilde{\xi} ^\top P_q \Tilde{\xi}$. The quantity $\mu >0$ is a free parameter. For each $q\in\mathcal Q_{\mathrm{tr}}$, the matrix $P_q$ is
symmetric and positive definite on $\operatorname{span}(V_q)$. To establish Theorem \ref{thm:ISS_like} we first prove two key technical lemmas. The first lemma characterizes positive definiteness on consistency spaces while the second bounds quadratic forms. To that end, for each $q \in \mathcal{Q}_{tr}$, let
$O_q \in \mathbb{R}^{n \times \dim(\text{span}(V_q))}$ denote a matrix whose
columns satisfy
\begin{equation}\label{conditionspan}
    O_q^\top O_q = I, \qquad
    \operatorname{range}(O_q) = \operatorname{span}(V_q).
\end{equation}
Since $\mathcal{C}_q = \operatorname{span}(V_q) - W_q b_q^a$ is in general
an affine (not linear) set, the matrix $O_q$ is taken
to satisfy $\operatorname{range}(O_q) = \operatorname{span}(V_q)$, the
linear direction space of $\mathcal{C}_q$. 
\begin{lem}
    \label{lem:PD}
   Let $q \in \mathcal{Q}_{\mathrm{tr}}$, and let
$O_q \in \mathbb{R}^{n \times \dim(\operatorname{span}(V_q))}$
have orthonormal columns satisfying \eqref{conditionspan}. Then, $P_q \succ 0$ on $\operatorname{span}(V_q)$ if and only if $O_q^\top P_q O_q \succ 0$.
\end{lem}
\noindent 
\textbf{Proof:} 
%
   Suppose first that $P_q \succ 0$ on $\operatorname{span}(V_q)$.
Let $z \in \mathbb{R}^{\dim(\operatorname{span}(V_q))}$ be nonzero.
Since $O_q$ has full column rank, $\eta := O_q z$ is nonzero, and by construction $\eta \in \operatorname{range}(O_q)
=
\operatorname{span}(V_q)$. Therefore, $z^\top O_q^\top P_q O_q z=
\eta^\top P_q \eta>0$. Since this holds for every nonzero $z$, it follows that $O_q^\top P_q O_q \succ 0$. Conversely, suppose that $O_q^\top P_q O_q \succ 0$.
Let $\eta \in \operatorname{span}(V_q)$ be nonzero.
Since the columns of $O_q$ form an orthonormal basis of
$\operatorname{span}(V_q)$, there exists a unique nonzero vector
$z \in \mathbb{R}^{\dim(\operatorname{span}(V_q))}$ such that $\eta = O_q z$. Hence, $\eta^\top P_q \eta=z^\top O_q^\top P_q O_q z>0$. Thus, $P_q \succ 0$ on $\operatorname{span}(V_q)$.  \hfill $\blacksquare$
\begin{lem}
    \label{lem: eig_vals}
    Let $q \in \mathcal{Q}_{\mathrm{tr}}$, and let
$O_q \in \mathbb{R}^{n \times \dim(\operatorname{span}(V_q))}$
have orthonormal columns satisfying \eqref{conditionspan}. If $P_q \succ 0$ on $\operatorname{span}(V_q)$, then $\lambda_{\min}\!\left(O_q^\top P_q O_q\right)\|\eta\|^2
\leq
\eta^\top P_q \eta
\leq
\lambda_{\max}\!\left(O_q^\top P_q O_q\right)\|\eta\|^2$, for all $\eta \in \operatorname{span}(V_q)$.
\end{lem}
\noindent 
 \textbf{Proof:} 
Let $\eta \in \operatorname{span}(V_q)$. Since the columns of
$O_q$ form an orthonormal basis of $\operatorname{span}(V_q)$,
there exists a unique vector
$z \in \mathbb{R}^{\dim(\operatorname{span}(V_q))}$ such that $\eta = O_q z$. Using $O_q^\top O_q = I$, we have $\|\eta\|^2=\|O_q z\|^2=z^\top O_q^\top O_q z=\|z\|^2$. Moreover, $\eta^\top P_q \eta=z^\top O_q^\top P_q O_q z$. Since $P_q \succ 0$ on $\operatorname{span}(V_q)$,
Lemma~3 implies that $O_q^\top P_q O_q \succ 0$. Therefore, by the standard Rayleigh quotient bounds,
\[
\lambda_{\min}\!\left(O_q^\top P_q O_q\right)\|z\|^2
\leq
z^\top O_q^\top P_q O_q z
\leq
\lambda_{\max}\!\left(O_q^\top P_q O_q\right)\|z\|^2.
\]
Using $\|z\|=\|\eta\|$ and
$\eta^\top P_q\eta=z^\top O_q^\top P_q O_q z$
gives the desired result.
\hfill $\blacksquare$

Using Lemmas \ref{lem:PD} and \ref{lem: eig_vals}, the following bounds on $V$ hold for all $x \in C$:
\begin{subequations}
\label{eqn: class_k_bounds}
\begin{align}
    &\min_{q\in\mathcal Q_{\mathrm{tr}}} \lambda_{\text{min}}(O^\top_qP_q O_q) \cdot \lvert x \rvert^2_{\mathcal{A}} \leq V(x), \\
    & V(x) \leq { \rm e}^{\mu N_0} \cdot \max_{q \in \mathcal{Q}_{\mathrm{tr}}} \lambda_{\text{max}} (O^\top_qP_q O_q) \cdot \lvert x \rvert^2_{\mathcal{A}}.
\end{align}    
\end{subequations}
Next, we evaluate the evolution in time of $V$ during flows and jumps. By considering the gradient of $V$ at a point $x \in C$, along an arbitrary direction of flow $f \in F(x)$, we obtain the bound:
\begin{align*}
    \langle\nabla V(x), f\rangle
    &\leq \bigl\langle {\rm e}^{\mu\tau}(P_q + P_q^\top)\tilde{\xi},\; 0,\; \mu {\rm e}^{\mu\tau}\tilde{\xi}^\top P_q \tilde{\xi}\bigr\rangle\\
    &\quad \cdot \bigl(\Pi_q^{\mathrm{diff}} A_q \tilde{\xi},\; 0,\; \sfrac{1}{\tau_d}\bigr)\\
    &= {\rm e}^{\mu\tau}\,\tilde{\xi}^\top\!\bigl(A_q^\top(\Pi_q^{\mathrm{diff}})^\top P_q
+P_q\Pi_q^{\mathrm{diff}}A_q\bigr)\tilde{\xi}\\
    &\quad + \frac{\mu}{\tau_d}\, V(x).
\end{align*}
Using Assumption~\ref{assump: matrices}, we get the inequality $\langle\nabla V(x), f\rangle \leq - {\rm e}^{\mu \tau} \Tilde{\xi}^\top M_q \Tilde{\xi} + \frac{\mu}{\tau_d} V(x)$ for all $x \in C$. Using Lemma \ref{lem: eig_vals} and the upper bound in \eqref{eqn: class_k_bounds}, we obtain 
\begin{equation}
    \label{eqn: flow_bound}
    \langle\nabla V(x), f\rangle \leq \left(\frac{\mu}{\tau_d} - \lambda\right)V(x), \quad \forall x \in C.
\end{equation}
We note that strong decrease of $V$ during flows holds precisely when $\frac{1}{\tau_d} < \sfrac{\lambda}{\mu}$. 
\smallskip
\par \noindent
To establish the decrease of the Lyapunov function during jumps, let $y \in D$, $g \in G(y)$ and consider the quantity
\begin{align*}
    V(g) &=  {\rm e}^{\mu \tau^+} (\Tilde{\xi}^+)^\top P_{q^+} (\Tilde{\xi}^+) \\
    & = {\rm e}^{\mu(\tau - 1)} \big(L_{q^+}\tilde{\xi} + \Delta_{q,q^+} \big)^\top P_{q^+} \big(L_{q^+}\tilde{\xi} + \Delta_{q,q^+}\big).
\end{align*}
Expanding this quadratic gives us three terms:
\begin{align*}
    V(g) = e^{\mu(\tau-1)} \Big[
    \tilde{\xi}^\top L_{q^+}^\top P_{q^+} L_{q^+} \tilde{\xi}
    &+ \underbrace{2\, \tilde{\xi}^\top L_{q^+}^\top P_{q^+}  \Delta_{q,q^+}}_{\text{(CT)}} \\
    &+ \Delta_{q,q^+}^\top P_{q^+} \Delta_{q,q^+}
\Big].
\end{align*}
We now apply Young's inequality with free weight $\theta > 0$ on the cross term (CT). That is, we use $2a^\top b \leq \theta a^\top a + \frac{1}{\theta} b^\top b$, with $a = P^{1/2}_{q^+}L_{q^+}\Tilde{\xi}, b = P^{1/2}_{q^+}\Delta_{q,q^+}$. This gives us
\begin{align*}
    V(g) \leq e^{\mu(\tau-1)} \Bigg[
    \underbrace{{\Big(1+\theta \Big)\, \tilde{\xi}^\top L_{q^+}^\top P_{q^+} L_{q^+} \tilde{\xi}}}_{\text{(I)}} \\ 
    + \underbrace{\left ( 1+\frac{1}{\theta} \right) \Delta_{q,q^+}^\top  P_{q^+} \Delta_{q,q^+}}_{\text{(II)}}
\Bigg].
\end{align*}
We now bound the terms (I) and (II) separately. 
For term (I), we first multiply and divide by the quantity $V(y)=
{\rm e}^{\mu\tau}\Tilde{\xi}^{\top}P_q\Tilde{\xi}.$ to get
\begin{align*}
    e^{\mu(\tau-1)}(1+\theta)\, \tilde{\xi}^\top L_{q^+}^\top P_{q^+} L_{q^+} \tilde{\xi}
&= e^{-\mu}(1+\theta)\, V(y) \cdot\\
&
\underbrace{\frac{\tilde{\xi}^\top L_{q^+}^\top P_{q^+} L_{q^+} \tilde{\xi}}{\tilde{\xi}^\top P_q \tilde{\xi}}}_{\displaystyle \le\, \gamma}.
\end{align*}
where we have used Lemma \ref{lem: eig_vals} for the upper bound above. For term (II), we have the bound
\begin{align*}
    e^{\mu(\tau-1)} \left(1+\frac{1}{\theta}\right) \Delta_{q,q^+}^\top  P_{q^+}  \Delta_{q,q^+}
\;&\le\;
e^{-\mu} e^{\mu N_0} \left(1+\frac{1}{\theta}\right) \cdot \\
&\Delta_{q,q^+}^\top  P_{q^+} \Delta_{q,q^+}. \\
& \leq e^{-\mu} e^{\mu N_0} \left ( 1 + \frac{1}{\theta} \right ) \alpha.  
\end{align*}
Combining the two bounds for the terms (I) and (II), we get $V(g) \leq \kappa V(y) + c$ for all $y \in D$ and $g \in G(y)$,
%
%
where $\kappa = {\rm e}^{-\mu}(1+\theta)\gamma$, and $c = {\rm e}^{-\mu}{\rm e}^{\mu N_0}\!\Bigl(1+\tfrac{1}{\theta}\Bigr)\alpha$. We observe that $\kappa < 1$ holds precisely under
$\theta < {\rm e}^{\mu}/\gamma - 1$, which is feasible if and only if $\mu > \ln\gamma$. 

For the purpose of completeness, we now present the remainder of the proof to obtain the $\mathcal{KL}$-bound
\eqref{eq: kl_bound}. Throughout, $\rho := \lambda -
\frac{\mu}{\tau_d} > 0$, and $\kappa$, $c$ are as defined in the proof.

\smallskip \noindent
Let $\mathrm{dom}(x) \subset \mathbb{R}_{\geq 0} \times \mathbb{N}$
be the hybrid time domain of the solution, and let $s_j :=
\inf\{t \geq 0 : (t,j) \in \mathrm{dom}(x)\}$ denote the continuous
time at which the $j$-th jump occurs. Set $s_0 := 0$, $\Delta_j:=s_j-s_{j-1}\geq0.$, and let $a_j := V(s_j, j)$ denote the value of
the Lyapunov function immediately \emph{after} the $j$-th jump. On
the interval $[s_{j-1}, s_j]$, the solution flows without jumping,
so the flow bound~\eqref{eqn: flow_bound} applies over the full
duration $\Delta_j$:
\begin{equation*}
    V(s_j^-, j-1) \leq {\rm e}^{-\rho \Delta_j}\, a_{j-1},
\end{equation*}
where $V(s_j^-, j-1)$ is the value immediately \emph{before} the
$j$-th jump. Applying the jump bound at
$(s_j, j)$ then gives the recursion
\begin{equation}\label{eq: A1}
    a_j \leq \kappa\, {\rm e}^{-\rho \Delta_j}\, a_{j-1} + c.
\end{equation}
The factor $\kappa < 1$ captures the contraction introduced by the
jump, while $c \geq 0$ accounts for the shift in equilibrium at each
mode transition.

\smallskip \noindent
Applying~\eqref{eq: A1} repeatedly from $j$ down to $0$ and using
$s_j = \sum_{i=1}^{j} \Delta_i$, we obtain $a_j \leq \kappa^j {\rm e}^{-\rho s_j} a_0
    + c\sum_{i=1}^{j} \kappa^{j-i} {\rm e}^{-\rho(s_j - s_i)}$.
Bounding
${\rm e}^{-\rho(s_j-s_i)} \leq 1$ and summing the geometric series
in $\kappa < 1$ yields
\begin{equation}\label{eq: A2}
    a_j \leq \kappa^j {\rm e}^{-\rho s_j} a_0 + \frac{c}{1-\kappa}.
\end{equation}

\smallskip \noindent
The bound~\eqref{eq: A2} holds only at jump times. To extend it to
all $(t,j) \in \mathrm{dom}(x)$, fix any $j$ and let $t \in [s_j,
s_{j+1})$, so that the solution is in its $j$-th flow interval with
no jumps occurring in $(s_j, t]$. Applying the flow
bound~\eqref{eqn: flow_bound} over the elapsed flow time $t - s_j$
gives $V(t,j) \leq {\rm e}^{-\rho(t - s_j)}\, a_j$. Substituting~\eqref{eq: A2} and using ${\rm e}^{-\rho(t-s_j)} \cdot
{\rm e}^{-\rho s_j} = {\rm e}^{-\rho t}$,
\begin{equation}\label{eq: A3}
    V(t,j) \leq {\rm e}^{-\rho t}\, \kappa^j\, V(0,0) +
    \frac{c}{1-\kappa},
\end{equation}
which holds for all $(t,j) \in \mathrm{dom}(x)$.

\smallskip \noindent
The transient term in~\eqref{eq: A3} decays in $t$ at rate $\rho$
and in $j$ at rate $-\ln\kappa > 0$. To obtain a single rate of
decay in the hybrid time $t + j$, let $\eta := \min\{\rho,\; -\ln\kappa\} > 0$. Then ${\rm e}^{-\rho t} \leq {\rm e}^{-\eta t}$ and $\kappa^j =
{\rm e}^{j\ln\kappa} \leq {\rm e}^{-\eta j}$, so that ${\rm e}^{-\rho t}\,\kappa^j \leq {\rm e}^{-\eta t}\,
    {\rm e}^{-\eta j} = {\rm e}^{-\eta(t+j)}$. Substituting into~\eqref{eq: A3},
\begin{equation}\label{eq: A4}
    V(t,j) \leq {\rm e}^{-\eta(t+j)}\, V(0,0) + \frac{c}{1-\kappa}.
\end{equation}

\smallskip \noindent
The bounds~\eqref{eqn: class_k_bounds} give $\underline{p}\,|x|_{\mathcal{A}}^2 \;\leq\; V(x) \;\leq\;
    {\rm e}^{\mu N_0}\overline{p}\,|x|_{\mathcal{A}}^2$. Applying the upper bound to
$V(0,0)$ and the lower bound to $V(t,j)$ in~\eqref{eq: A4},
\begin{equation*}
    \underline{p}\,|x(t,j)|_{\mathcal{A}}^2 \leq
    {\rm e}^{-\eta(t+j)}\,{\rm e}^{\mu N_0}\overline{p}\,
    |x(0,0)|_{\mathcal{A}}^2 + \frac{c}{1-\kappa}.
\end{equation*}
Taking square roots and applying $\sqrt{a + b} \leq \sqrt{a} +
\sqrt{b}$,
\begin{equation*}
    |x(t,j)|_{\mathcal{A}} \leq
    \sqrt{\frac{{\rm e}^{\mu N_0}\overline{p}}{\underline{p}}}\,
    {\rm e}^{-\eta(t+j)/2}\,|x(0,0)|_{\mathcal{A}}
    + \sqrt{\frac{c}{\underline{p}\,(1-\kappa)}}.
\end{equation*}
Substituting $c = {\rm e}^{-\mu}{\rm e}^{\mu N_0}(1+\tfrac{1}{\theta})\alpha$
and $1 - \kappa = {\rm e}^{-\mu}({\rm e}^{\mu} - (1+\theta)\gamma)$
into the residual term,
\begin{equation*}
    \sqrt{\frac{c}{\underline{p}\,(1-\kappa)}}
    = \sqrt{\frac{{\rm e}^{\mu N_0}\bigl(1+\tfrac{1}{\theta}\bigr)\alpha}
             {\underline{p}\,\bigl({\rm e}^{\mu} -
             (1+\theta)\gamma\bigr)}}
    = \nu(\alpha),
\end{equation*}
which completes the proof of~\eqref{eq: kl_bound}. \hfill $\blacksquare$
\subsection{Proof of Theorem 2}
To characterize the limiting behavior of system \eqref{eqn: equiv_2}, we consider a HDS for which $\tau_d\to\infty$, i.e., only a finite number of jumps is allowed:
\begin{subequations}
    \label{eqn: ideal_hybrid}
    \begin{align}
    \dot{\xi}=\Pi^{\text{diff}}_{q} (A_q \xi +& b_q),~~~\dot{q}=0,~~~\dot{\tau}=0,
        ~ x \in C  \\[2mm]
        \begin{bmatrix}
            (\xi ^+, q^+)^\top \\
            \tau^+
        \end{bmatrix} &\in 
        \begin{bmatrix}
        \{(\Pi_{z}\xi,z):z\in\mathcal{Q}_{\text{tr}}\backslash\{q\}\}\\
            \tau - 1
        \end{bmatrix},  ~ x \in D.
    \end{align}
\end{subequations} 
For this system, and for each $q \in \mathcal{Q}_{\text{tr}}$ and $j \in \mathbb{Z}_{\geq 0}$, define the set  $\mathcal{A}_{j,q}=\overline{R \left( \left(\left ( \{\xi^*_q\} + (j+1)^{-1}\mathbb{B} \right )\cap\mathcal{C}_q\right) \times \{q\} \times [0,1]\right )}$, as well as the sets
\begin{align}\label{eqn:omega_limit_set}
        \mathcal{A}_q  \coloneqq \bigcap_{j \in \mathbb{Z}_{\geq 0}} \mathcal{A}_{j,q},~~~~\text{and}~~~~~  \mathcal B\coloneqq
\bigcup_{q\in\mathcal Q_{\mathrm{tr}}}\mathcal A_q
\end{align}
Adapting the proof of \cite[Proposition 4]{baradaran2020omega} to the DAE case, it can be shown that the set $\mathcal{B}$ is the Omega-limit set of System \eqref{eqn: ideal_hybrid} starting from a sufficiently large compact set $K$ that contains the set $\left ( \bigcup_{q \in \mathcal{Q}_{\mathrm{tr}}} \{\xi^*_q\} \times \{q\} \right ) \times [0,N_0]$ in its interior. In what follows, we provide an explicit characterization of this Omega-limit set using the solution formula of \eqref{eqn: dae_sol}. 
%
%
%
\begin{prop}
    \label{prop: ols_dae}
    Assume $N_0=1$ and suppose that Assumptions
\ref{assumpregdae}, \ref{assump: regularity},
and \ref{assump: matrices} hold. Then, the sets $\mathcal{B}$ and
    \begin{align}
\mathcal{A}
:=
\bigcup_{q\in\mathcal{Q}_{\mathrm{tr}}}
\Bigg[
&\{\xi_q^\star\}\times\{q\}\times[0,1]
\\
&{}\cup
\bigcup_{\substack{r\in\mathcal{Q}_{\mathrm{tr}}\notag\\ r\neq q}}
\overline{
\left\{
\Theta_r^t\bigl(\Pi_r(\xi_q^\star)\bigr)
:
t\geq 0
\right\}}
\times\{r\}\times\{0\}
\Bigg].
\end{align}
    are equal. 
\end{prop}

\vspace{0.1cm}\noindent 
\textbf{Proof:} Let $K^j_{q}:=  (\left(\{\xi^*_q\} + (j+1)^{-1} \mathbb{B})\cap\mathcal{C}_q\right) \times \{q\} \times [0, 1]$ and $K^\infty_{q}:= \{\xi^*_q\} \times \{q\} \times [0, 1]$. Note that $\mathcal{A}_{j,q} = \overline{R(K^j_q)}$. Since $K^\infty_{q}\subset K^j_{q}$, for all $j\geq 0$, we have $\overline{R\left (K^\infty_{q}\right )} \subseteq \mathcal{A}_{j,q}$ for all $j\in\mathbb{Z}_{\geq 0}$. Hence $\overline{R\left (K^\infty_{q}\right )} \subseteq \bigcap_{j\in\mathbb{Z}_{\geq 0}} \mathcal{A}_{j,q} = \mathcal{A}_{q} \subset \mathcal{B}$. To show the other direction, we note that since $N_0=1$, system \eqref{eqn: ideal_hybrid} can experience at most one jump. Therefore, for any $\phi \in \mathcal{S}(K^j_q)$, one of the following must hold true: 1) \text{dom}($\phi$) = $[0, \infty) \times \{0\}$; 2) \text{dom}($\phi$) = $[0,t_1] \times \{0\} \cup [t_1,\infty) \times \{1\}$, for some $t_1 \geq 0$.  Let us consider the first possibility. It is clear that 
    \begin{equation*}
        \phi(t,0) = (\Theta_q^t(\xi(0,0)), q,\tau(0,0)), \, \forall (t,0) \in \text{dom}(\phi),
    \end{equation*}
    where $\xi(0,0) \in \{\xi^*_q\} + (j+1)^{-1} \mathbb{B}$, $q\in\mathcal Q_{\mathrm{tr}}$, and $\tau(0,0) \in [0,1]$. Since $\xi^*_q$ is globally asymptotically stable for the $q^{\text{th}}$ mode of System \eqref{eqn: ideal_hybrid}, there exists a constant $c_1 >0$ such that $\Theta^t_q(\xi(0,0)) \in \left (\{\xi^*_q\} + c_1(j+1)^{-1} \mathbb{B}  \right )\cap\mathcal{C}_q$, for all $(t,0) \in \text{dom}(\phi)$. Therefore,
    \begin{align*}
        \phi(t,i)&\in ((\{\xi^*_q\} + c_1(j+1)^{-1} \mathbb{B})\cap\mathcal{C}_q) \times \{q\} \times [0, 1]
    \end{align*}
    for all $(t,i)\in \text{dom}(\phi)$.
    \par Now, we consider the second possibility. For all $(t,0) \in \text{dom}(\phi)$, we have  $\phi(t,0) = (\Theta_q^t(\xi(0,0)), q,\tau(0,0))$. Now, for $(t,1) \in \text{dom}(\phi)$ we have 
    \begin{equation*}
        \phi(t,1) = (\Theta^{t-t_1}_{q^+} \circ \Pi_{q^+} \circ \Theta_q^{t_1}(\xi(0,0)), q^+, 0). 
    \end{equation*}
    Similar to the previous possibility, there exists a constant $c_1 >0$ such that  $\Theta^t_q(\xi(0,0)) \in \left (\{\xi^*_q\} + c_1(j+1)^{-1} \mathbb{B}  \right )\cap\mathcal{C}_q$, for all $(t,0) \in \text{dom}(\phi)$. Since \textcolor{black}{$L_{q^+}$ is the linear part of $\Pi_{q^+}$, the map $\Pi_{q^+}$ is Lipschitz with constant $l_{q^+} := \|L_{q^+}\|_2$}. Therefore, the image of
$\{\xi_q^\star\}+c_1(j+1)^{-1}\mathbb B$
under $\Pi_{q^+}$ is contained in
$\{\Pi_{q^+}(\xi_q^\star)\}
+c_1l_{q^+}(j+1)^{-1}\mathbb B.
$. Thus, for all $(t,1) \in \text{dom}(\phi)$ we have
    \begin{align*}
\Theta_{q^+}^{t-t_1}
\circ\Pi_{q^+}
\circ&\Theta_q^{t_1}(\xi(0,0))
\in
\Bigl\{
\Theta_{q^+}^{t-t_1}(\xi)
:\,\notag\\
&\bigl|\xi-\Pi_{q^+}(\xi_q^\star)\bigr|
\leq c_1l_{q^+}(j+1)^{-1}
\Bigr\}.
\end{align*}
    %
    Since $\xi^*_{q^+}$ is globally exponentially stable for the $(q^+)^{\text{th}}$ mode of System \eqref{eqn: ideal_hybrid} (a direct consequence of Assumption \ref{assump: matrices}), \textcolor{black}{there exists $c_2 \geq 1$ such that $\|e^{A_{q^+}^{\mathrm{diff}} t}\| \leq c_2$ for all $t \geq 0$ and $q^+ \in \mathcal{Q}_{\text{tr}}$ (see, e.g.,~\cite[Thm.~4.4]{trenn2012switched}).} \textcolor{black}{Consequently}, there exists a constant $c_2 >0$ such that for all $t \geq 0$ and $q \in \mathcal{Q}_{\text{tr}}$,
       \begin{equation*}
\left|
\Theta_q^t(\xi)-\Theta_q^t(\xi')
\right|
\leq
c_2|\xi-\xi'|,
\qquad
\forall\,\xi,\xi'\in\mathcal C_q,
\quad t\geq0.
\end{equation*}
    As a result we get 
    \begin{align}
        \label{eqn: containment}
        \{ &\Theta_{q^+}^{t-t_1}(\xi) \, : \, \lvert \xi - \Pi_{q^+}(\xi^*_q) \vert  \leq c_1 \cdot l_{q^+} (j+1)^{-1}\} \subset \nonumber \\
        & \{\Theta_{q^+}^{t-t_1}(\Pi_{q^+} \xi^*_q)\} + c_1 \cdot l_{q^+} \cdot c_2 (j+1)^{-1} \mathbb{B}. 
    \end{align}
    It is clear now that for all $(t,1) \in \text{dom}(\phi)$,
    %
    \begin{align*}
\phi(t,1)
\in
\bigcup_{q^+\in\mathcal Q_{\mathrm{tr}}\setminus\{q\}}
\Bigl(
&
\bigl\{
\Theta_{q^+}^{t-t_1}
\bigl(\Pi_{q^+}(\xi_q^\star)\bigr)
\bigr\}
\\
&{}
+c_1c_2l(j+1)^{-1}\mathbb B
\Bigr)
\times\{q^+\}\times\{0\}.
\end{align*}
    where $l \coloneqq \max_{r \in \mathcal{Q}_{\text{tr}}} \lVert \textcolor{black}{L_r} \rVert$.
    From the above arguments:
    \begin{align*}
        R(K_q^\infty)
        ={}&
        \{\xi_q^\ast\}\times\{q\}\times[0,1]
        \nonumber\\
        &\cup
        \bigcup_{r\in\mathcal{Q}_{\text{tr}}\setminus\{q\}}
        \{\Theta_r^t(\Pi_r(\xi_q^\ast)):t\geq0\}
        \times\{r\}\times\{0\}.
    \end{align*}
    Since each subsystem has an exponentially stable equilibrium point, we also have
    \begin{align}
        &\overline{R(K_q^\infty)}
        ={}
        \{\xi_q^\ast\}\times\{q\}\times[0,1]
        \nonumber\\
        &\cup
        \bigcup_{r\in\mathcal{Q}_{\text{tr}}\setminus\{q\}}
        \left(
        \{\Theta_r^t(\Pi_r(\xi_q^\ast)):t\geq0\}
        \cup\{\xi_r^\ast\}
        \right)
        \times\{r\}\times\{0\}.
        \label{eq:closureR0Kinf}
    \end{align}
    We now show that $\mathcal{A}_q=
    \overline{R(K_q^\infty)}$. Indeed, since $K_q^\infty\subset K_q^j$ for every
$j\in\mathbb{Z}_{\geq0}$, it follows that $R(K_q^\infty)
    \subseteq
    R(K_q^j)$, and therefore $\overline{R(K_q^\infty)}
    \subseteq
    \overline{R(K_q^j)}
    =
    \mathcal{A}_{j,q}$ for every $j\in\mathbb{Z}_{\geq0}$. Thus,
\begin{equation}
    \overline{R(K_q^\infty)}
    \subseteq
    \bigcap_{j\in\mathbb{Z}_{\geq0}}\mathcal{A}_{j,q}
    =
    \mathcal{A}_q.
\label{eq:first-inc}
\end{equation}
To show the other inclusion, we define the constant $c_3 \coloneqq \max \{c_1, c_1 \cdot c_2 \cdot l\}$. Using the containment in \eqref{eqn: containment}, it is clear that  $R(K_q^j)
    \subseteq
    R(K_q^\infty)
    +c_3(j+1)^{-1}\mathbb{B}$. Taking the closure on both sides gives us,  $\mathcal{A}_{j,q}
    =
    \overline{R(K_q^j)}
    \subseteq
    \overline{
    R(K_q^\infty)
    +c_3(j+1)^{-1}\mathbb{B}}$. Since $R(K_q^\infty)
    \subseteq
    \overline{R(K_q^\infty)}$ we further obtain $\mathcal{A}_{j,q}
    \subseteq
    \overline{R(K_q^\infty)}
    +c_3(j+1)^{-1}\mathbb{B}$, where we have used the fact that
$\overline{R(K_q^\infty)}$ is closed and the ball is compact.
Therefore,
\begin{align}
    \mathcal{A}_q
    &=
    \bigcap_{j\in\mathbb{Z}_{\geq0}}\mathcal{A}_{j,q}
    \nonumber\subseteq
    \bigcap_{j\in\mathbb{Z}_{\geq0}}
    \left(
    \overline{R(K_q^\infty)}
    +c_3(j+1)^{-1}\mathbb{B}
    \right)
    \nonumber\\
    &=
    \overline{R(K_q^\infty)}.
\label{eq:second-inc}
\end{align}
The inclusions in \eqref{eq:first-inc} and \eqref{eq:second-inc} give us
$\mathcal{A}_q=\overline{R(K_q^\infty)}$. 

Finally, taking the union over $q\in\mathcal{Q}_{\text{tr}}$ and using
\eqref{eq:closureR0Kinf}, we obtain
\begin{align*}
\mathcal{B}
&=
\bigcup_{q\in\mathcal{Q}_{\text{tr}}}
\overline{R(K_q^\infty)}=\bigcup_{q\in\mathcal{Q}_{\text{tr}}}
\Bigg[
\{\xi_q^\ast\}\times\{q\}\times[0,1]
\\[-1mm]
&\left.
{}\cup
\bigcup_{r\in\mathcal{Q}_{\text{tr}}\setminus\{q\}}
\left(
\{\Theta_r^t(\Pi_r(\xi_q^\ast)):t\geq0\}
\cup\{\xi_r^\ast\}
\right)
\times\{r\}\times\{0\}
\right] \\
&= \mathcal{A}.
\end{align*}
This completes the proof.  \hfill $\blacksquare$

\vspace{0.1cm}
\textbf{Proof of Theorem 2:} The proof follows similar steps as in \cite[Theorem 1]{baradaran2020omega} but uses the sharper characterization of the set $\mathcal{A}$ in Proposition \ref{prop: ols_dae}. It is clear that the set $C$ and $D$ are closed since $\mathcal{Q}_{\text{tr}}$ is finite, the interval $[0,1]$ is closed, and any finite-dimensional affine subspace is also closed. Moreover, the maps $F$ and $G$ are OSC, LB and the map $F$ is convex-valued. Therefore, system \eqref{eqn: eqiv_HI} satisfies the hybrid basic conditions. Let $K$ be chosen sufficiently large such that it contains in its interior the set $\mathcal{A}$. Following the same arguments as in \cite{baradaran2020omega}, we have that $\Omega(K)= \mathcal{A}$. From \cite[Corollary 7.7]{goebel2012hybrid}, we obtain that the set $\mathcal{A}$ is asymptotically stable for system \eqref{eqn: ideal_hybrid}. Since the preceding argument applies to every sufficiently large
compact set $K$, and such compact sets can be chosen to contain any
prescribed compact subset of the admissible state space, the basin
of attraction of $\mathcal A$ is the entire admissible state space.
Therefore, $\mathcal A$ is globally asymptotically stable for
system~\eqref{eqn: ideal_hybrid}. Now consider the inflated hybrid system $\mathcal{H}_{\epsilon}$, given by
\begin{align}\label{eqn:inflation_1}
    x \in C_\epsilon,~~\Dot{x}\in F_\epsilon(x),~~~~~~
    x \in D_\epsilon,~~x^+\in G_\epsilon(x),  
\end{align}
where $\epsilon=\frac{1}{\tau_d}$, and where the data of \eqref{eqn:inflation_1} is constructed from the data of $\mathcal{H}_0$ as in \cite[Sec. 6]{teel2013lyapunov}. In particular, for $\epsilon=1/\tau_d$, the data of the system satisfy $C\subseteq C_\epsilon$, $D\subseteq D_\epsilon$, and, for every admissible state, we have $F_{\tau_d}(x)\subseteq F_\epsilon(x)$, and $G_{\tau_d}(x)\subseteq G_\epsilon(x)$. Consequently, every solution of the original system is also a solution of the inflated system
$\mathcal H_\epsilon$, and, by \cite[Thm. 7.21]{goebel2012hybrid}, the inflated hybrid system $\mathcal{H}_{\epsilon}$ renders the set $\mathcal{A}$ semi-globally practically asymptotically stable as $\epsilon\to0^+$. \hfill $\blacksquare$
\section{Proofs for the Inference-Phase Switching Case}
\textbf{Proof of Theorem \ref{thm: shds_recurrence}:} Consider the function $V(x) \;=\; e^{-\mu\tau/N_0}\,\tilde\xi^\top P_q \tilde\xi$, where $\mu > 0$ is a free parameter and for each $q \in \mathcal{Q}_{\text{in}}$, the matrices $P_q$ are symmetric and positive definite on
$\operatorname{span}(V_q)$.  For each $q' \in \mathcal{Q}_{\text{in}}$,
define the integral
\begin{align}
\label{eq: I_def}
    \tilde{I}_{q'}(\mu):=
    \mathbb{E}_{v_2\mid v_1=q'}\!\left[e^{-\mu v_2/N_0}\right]= \int_0^{N_0} f_{q'}(v_2)e^{-\mu v_2/N_0}dv_2,
\end{align}
where $f_{q'}$ is the probability density for $v_2$.
It can be seen that the integral $\tilde{I}_{q'}(\mu)$ is strictly positive and
monotonically decreasing in $\mu$, with $\tilde{I}_{q'}(0)=1$ and
$\tilde{I}_{q'}(\mu)\to 0$ as $\mu\to\infty$ for all three modes. For the convenience of the reader, we provide explicit expressions for $\tilde{I}_{q'}(\mu)$ in Appendix D.

Next, we show that $V$ is a
Lagrange--Foster function for the SHDS~\eqref{eq: shds} in the sense
of~\cite[Sec.~4.4.1]{teel2013lyapunov}. We begin by showing the class $\mathcal{G}_\infty$ bounds in \cite[(16a), (16b)]{teel2013lyapunov}. Indeed, note that since $\mathcal{A}_{\text{in}} \subset \mathcal{O}_\varepsilon \subset \Bar{\mathcal{O}_\varepsilon}$, we have
\begin{equation}
\label{eqn: distance}
\lvert x\rvert_{\Bar{\mathcal{O}_\varepsilon}} \leq \lvert x \rvert_{\mathcal{A}_{\text{in}}}, \, \forall x \in C \cup D.    
\end{equation}
\par \noindent
Since $\tau \in [0,N_0]$, we have $e^{-\mu\tau/N_0} \in [e^{-\mu}, 1]$,
giving
\begin{equation}\label{eq: shds_sandwich}
    e^{-\mu}\,\underline{p}\,|x|^2_{\mathcal{A}_{\text{in}}}
    \;\leq\; V(x) \;\leq\;
    \overline{p}\,|x|^2_{\mathcal{A}_{\text{in}}},
    \quad \forall x \in C\cup D,
\end{equation}
where $\underline{p}$ and $\overline{p}$ are defined as in
Section~\ref{sec:formulation}, but relative to $\mathcal{Q}_{\text{in}}$. Both bounding functions are class
$\mathcal{G}_\infty$. Using $\eqref{eqn: distance}$ and the lower bound in \eqref{eq: shds_sandwich}, we satisfy \cite[(16a)]{teel2013lyapunov}. The upper bound of \eqref{eq: shds_sandwich} satisfies \cite[(16b)]{teel2013lyapunov} with $\varpi(x) = \lvert x\rvert_{\mathcal{A}_{\text{in}}}$.

\smallskip \noindent
To show decrease during flows, let $x \in C$ and consider an arbitrary direction of flow $f \in F(x)$, with
worst-case $\dot\tau = -\frac{1}{\tau_d}$. We have,  $\langle \nabla V(x), f\rangle
    = e^{-\frac{\mu\tau}{N_0}}\bigl(2\tilde\xi^\top P_q\Pi_q^{\mathrm{diff}}A_q\tilde\xi\bigr)+ \frac{\mu e^{-\frac{\mu\tau}{N_0}}}{\tau_d \cdot  N_0}\tilde\xi^\top P_q\tilde\xi$. Using Assumption~\ref{assump: flow_LMI}, we get 
\begin{align}
    \label{eq: flow_final}
    \langle \nabla V(x), f\rangle \; &\leq \;
    e^{-\mu\tau/N_0}\left(-\tilde\xi^\top \tilde{N}_q\tilde\xi \right)\leq-\lambda V(x),
\end{align}
where $\lambda:=\; \frac{\underline{m}}{\overline{p}} > 0$ and $\underline{m} := \min_{q\in\mathcal{Q}_{\text{in}}}\lambda_{\min}(O_q^\top \tilde{N}_q O_q)$. Now, let $x \in D$ and $g \in G(x, v_1, v_2)$. At a jump, we have 
$\tilde\xi^+ = L_{v_1}\tilde\xi + \Delta_{q,v_1}$ and $\tau^+ = v_2$, and $V$ satisfies  
\begin{equation*}
    V(g) = e^{-\mu v_2/N_0}\,
    (L_{v_1}\tilde\xi + \Delta_{q,v_1})^\top P_{v_1}
    (L_{v_1}\tilde\xi + \Delta_{q,v_1}).
\end{equation*}
Using the tower property, we get
\begin{align*}
\mathbb{E}[V(g)\mid x]&=
\mathbb{E}_{v_1}\!\bigg[
\mathbb{E}_{v_2\mid v_1}\!\big[
e^{-\mu v_2/N_0}
\left(L_{v_1}\tilde{\xi}+\Delta_{q,v_1}\right)^\top
P_{v_1}\\
&\cdot
\left(L_{v_1}\tilde{\xi}+\Delta_{q,v_1}\right)
\big]
\bigg]
\\
&=\mathbb{E}_{v_1} \Big [ \mathbb{E}_{v_2 \lvert v_1} [e^{-\mu v_2/N_0}] \cdot  (L_{v_1}\tilde\xi + \Delta_{q,v_1})^\top P_{v_1} \\
    & \cdot (L_{v_1}\tilde\xi + \Delta_{q,v_1})    \Big] \\[2mm]
    &= \sum_{q'\in\mathcal{Q}_{\text{in}}} p_{q'} \tilde{I}_{q'}(\mu) \cdot
       (L_{q'}\tilde\xi + \Delta_{q,q'})^\top P_{q'} \\
      & \cdot(L_{q'}\tilde\xi + \Delta_{q,q'}).
\end{align*}
Expanding and applying Young's inequality with the weight $\theta > 0$, we get $(L_{q'}\tilde\xi + \Delta_{q,q'})^\top P_{q'}
    (\textcolor{black}{L_{q'}}\tilde\xi + \Delta_{q,q'})
\leq (1+\theta)\,\tilde\xi^\top L_{q'}^\top P_{q'} L_{q'}\tilde\xi
    + \Bigl(1+\tfrac{1}{\theta}\Bigr)\Delta_{q,q'}^\top P_{q'}\Delta_{q,q'}$. Finally, we multiply the above inequality with $p_{q'} \tilde{I}_{q'}(\mu)$, sum over the modes $q' \in \mathcal{Q}_{\text{in}}$, and invoke the second inequality in 
Assumption~\ref{assump: flow_LMI} on the first term to get
\begin{equation}\label{eq: jump_ineq_shds}
    \mathbb{E}[V(g)\mid x] \;\leq\;
    \kappa^*\tilde\xi^\top P_q\tilde\xi +
    \bigl(1+\tfrac{1}{\theta}\bigr)\alpha_\Gamma(\mu).
\end{equation}
Since $\tau = 0$ at a jump, we have $V(x)= e^0\,\tilde\xi^\top P_q\tilde\xi= \tilde\xi^\top P_q\tilde\xi$. So, $\kappa^*\tilde\xi^\top P_q\tilde\xi = \kappa^* V(x)$, giving us the final bound
\begin{equation}\label{eq: jump_final}
    \mathbb{E}[V(g)|x]\leq \kappa^* V(x) + c, \forall x \in D,
\end{equation}
where $c = \bigl(1+\tfrac{1}{\theta}\bigr)\alpha_\Gamma(\mu)$.  Next, define the truncated flow and jump sets as $C^\circ := C \setminus \mathcal{O}_\varepsilon$, $D^\circ := D \setminus \mathcal{O}_\varepsilon$. For $x \in C^\circ \cup D^\circ$, we have $|x|_{\mathcal{A}_{\text{in}}}
\ge r(\alpha_\Gamma(\mu))+\varepsilon
> r(\alpha_\Gamma(\mu)),
$
and therefore $(1-\kappa^\star)V(x)-c>0$. Hence, on the truncated system, we may take $\rho(x)=
\min\left\{
\lambda V(x),
\left[(1-\kappa^\star)V(x)-c\right]_+
\right\}
>0$, where $[s]_+=\max\{0,s\}$,
to obtain the following inequalities via
$\eqref{eq: flow_final}$ and $\eqref{eq: jump_final}$: $\langle \nabla V(x), f\rangle \leq - \rho(x)$ for all $x \in C^\circ, f \in F(x)$; and $\mathbb{E}[V(g)|x] \leq V(x) - \rho(x)$ for all $x \in D^\circ.$ To show that $\rho \in \mathcal{PD}(\mathcal{O}_\varepsilon)$, define the compact set $R \coloneqq \{x \in C\cup D :
    |x|_{\mathcal{A}_{\text{in}}} \leq r\!\left(\alpha_\Gamma(\mu)\right)\} $ and note that $(R + \frac{\varepsilon}{2} \mathbb{B})\cap (C\cup D) \subset \mathcal{O}_\varepsilon$. Now, we only need to show that
$\rho(x)>0$ for all $x\in(C\cup D)\setminus R$.
Hence,
$\lvert x\rvert_{\mathcal{A}_{\mathrm{in}}}
>r(\alpha_\Gamma(\mu))$.
To that end, choose any
$x\in(C\cup D)\setminus R$ and consider the following bounds: $V(x)\ge e^{-\mu}\underline{p}|x|_{\mathcal{A}_{\text{in}}}^2 >e^{-\mu}\underline{p} r(\alpha_\Gamma(\mu))^2=\frac{c}{1-\kappa^{\star}} > 0$, using $r(\alpha_\Gamma(\mu))^2 = c/(e^{-\mu}\underline{p}(1-\kappa^{\star}))$. Hence, $\rho(x)>0$ for all
$x\in(C\cup D)\setminus R$, implying that
$\rho\in\mathcal{PD}(\mathcal{O}_\varepsilon)$.
\vspace{0.1cm}
\par \noindent
Therefore, $V$ is a Lagrange-Foster function and by Theorem \cite[4.4]{teel2013lyapunov} the set $\mathcal{O}_\varepsilon$ is uniformly globally recurrent for the SHDS \eqref{eq: shds}. \hfill $\blacksquare$
\section{Explicit Computation of $\tilde{I}_{q'}(\mu)$}
\label{appen:integrals}
In this appendix, we provide explicit expressions for the
mode-dependent normalized timer-reset integral
\(\tilde{I}_{q'}(\mu)\) for the Idle and Prefill modes, together
with a general expression for the Decode mode.

\subsection*{Idle: Truncated Exponential}

The normalized Idle timer-reset variable follows an
\(\mathrm{Exp}(\bar{\lambda}_r)\) distribution truncated to
\([0,N_0]\), with density
\[
f_{\mathrm{id}}(v)
=
\frac{\bar{\lambda}_r e^{-\bar{\lambda}_r v}}
{1-e^{-\bar{\lambda}_r N_0}},
\qquad
v\in[0,N_0].
\]
Using the definition of the timer-reset integral and combining the
exponentials gives
\begin{equation}
\label{eqn:idle_dwell}
\tilde{I}_{\mathrm{id}}(\mu)
=
\frac{\bar{\lambda}_r N_0}
{\mu+\bar{\lambda}_r N_0}
\frac{
1-e^{-(\mu/N_0+\bar{\lambda}_r)N_0}
}{
1-e^{-\bar{\lambda}_r N_0}
}.
\end{equation}
The expression is finite for every \(\mu>0\) and satisfies
\(\tilde{I}_{\mathrm{id}}(0)=1\). Moreover,
\[
\tilde{I}_{\mathrm{id}}(\mu)
\sim
\frac{\bar{\lambda}_r N_0}
{1-e^{-\bar{\lambda}_r N_0}}
\frac{1}{\mu}
=
O(\mu^{-1}),
\qquad
\mu\to\infty.
\]
\subsection*{Prefill: Truncated Gamma}
The normalized Prefill timer-reset variable follows
$\mathrm{Gamma}(\alpha_{\mathrm{pre}},
\bar{\beta}_{\mathrm{pre}})$, truncated to $[0,N_0]$,
with density
\[
f_{\mathrm{pre}}(v)
=
\frac{
\bar{\beta}_{\mathrm{pre}}^{\alpha_{\mathrm{pre}}}
v^{\alpha_{\mathrm{pre}}-1}
e^{-\bar{\beta}_{\mathrm{pre}}v}
}{
\gamma\!\left(
\alpha_{\mathrm{pre}},
\bar{\beta}_{\mathrm{pre}}N_0
\right)
},
\qquad v\in[0,N_0].
\]
where $\gamma(\alpha,x)$ denotes the  lower incomplete Gamma
function \cite[Eq.~6.5.2]{abramowitz1964}. Substituting into \eqref{eq: I_def} and combining the exponentials
$e^{-\bar{\beta}_{\mathrm{pre}} v}$ and $e^{-\mu v/N_0}$ into a single term
$e^{-sv}$ with $s := \bar{\beta}_{\mathrm{pre}} + \mu/N_0$,
\begin{equation}
\tilde{I}_{\mathrm{pre}}(\mu) = \frac{\bar{\beta}_{\mathrm{pre}}^{\alpha_{\mathrm{pre}}}}
{\gamma(\alpha_{\mathrm{pre}}, \bar{\beta}_{\mathrm{pre}} N_0)}
\int_0^{N_0} v^{\alpha_{\mathrm{pre}}-1} e^{-sv}\, dv.
\end{equation}
Substituting $u = sv$ in the remaining integral gives the exact identity
\[
\int_0^{N_0} v^{\alpha_{\mathrm{pre}}-1} e^{-sv}\, dv
= s^{-\alpha_{\mathrm{pre}}} \gamma(\alpha_{\mathrm{pre}}, sN_0),
\]
so that
\begin{align*}
\tilde{I}_{\mathrm{pre}}(\mu) &= \frac{\bar{\beta}_{\mathrm{pre}}^{\alpha_{\mathrm{pre}}}}
{\gamma(\alpha_{\mathrm{pre}}, \bar{\beta}_{\mathrm{pre}} N_0)}\,
s^{-\alpha_{\mathrm{pre}}}\, \gamma(\alpha_{\mathrm{pre}}, sN_0).
\end{align*}
As $\mu \to \infty$, $s \to \infty$ and hence $sN_0 \to \infty$. Since
$\gamma(\alpha,x) \to \Gamma(\alpha)$ as $x \to \infty$, the incomplete Gamma
factor $\gamma(\alpha_{\mathrm{pre}}, sN_0) \to \Gamma(\alpha_{\mathrm{pre}})$,
a fixed constant. The remaining $\mu$-dependence is entirely in $s^{-\alpha_{\mathrm{pre}}}$,
and $s \sim \mu/N_0$ for large $\mu$, giving
\begin{align*}
\tilde{I}_{\mathrm{pre}}(\mu) &\sim
\frac{\bar{\beta}_{\mathrm{pre}}^{\alpha_{\mathrm{pre}}}\,\Gamma(\alpha_{\mathrm{pre}})}
{\gamma(\alpha_{\mathrm{pre}}, \bar{\beta}_{\mathrm{pre}} N_0)}\,
\left(\frac{N_0}{\mu}\right)^{\alpha_{\mathrm{pre}}}= O\!\left(\mu^{-\alpha_{\mathrm{pre}}}\right)
\end{align*}
as $\mu \to \infty$. For $\alpha_{\mathrm{pre}}>1$, the Prefill timer-reset integral
decays faster with respect to $\mu$ than the corresponding
Idle integral, provided the latter has a decay rate of
$\Theta(\mu^{-1})$

\subsection*{Decode: General Normalized Distribution on $[0,N_0]$}

The normalized Decode timer-reset variable has a general density
$f_{\mathrm{dec}}$ on $[0,N_0]$, obtained by scaling the corresponding
physical-time Decode distribution by $1/\tau_d$. Possible models
include a truncated Pareto or log-$t$ distribution. We assume that
there exist constants $C,\delta>0$ and $\epsilon\geq0$ such that
\[
    f_{\mathrm{dec}}(v)\leq Cv^\epsilon,
    \qquad v\in[0,\delta],
\]
and that $f_{\mathrm{dec}}$ is bounded by some $M<\infty$ on
$[\delta,N_0]$. No specific parametric family is otherwise assumed.
The corresponding normalized timer-reset integral is
\begin{equation}
\label{eqn:decode_dwell}
\tilde{I}_{\mathrm{dec}}(\mu) = \int_0^{N_0} f_{\mathrm{dec}}(v)\,
e^{-\mu v/N_0}\, dv.
\end{equation}
Splitting the integral at $\delta$ we obtain:
\[
\tilde{I}_{\mathrm{dec}}(\mu) = \underbrace{\int_0^{\delta} f_{\mathrm{dec}}(v)
e^{-\mu v/N_0}\, dv}_{=: J_1(\mu)}
+ \underbrace{\int_{\delta}^{N_0} f_{\mathrm{dec}}(v) e^{-\mu v/N_0}\, dv}_{=: J_2(\mu)}.
\]
For the second term, since $v\geq\delta$ on $[\delta,N_0]$, we have $e^{-\mu v/N_0}\leq e^{-\mu\delta/N_0}$.Consequently, $J_2(\mu)
\leq M(N_0-\delta)e^{-\mu\delta/N_0}$. For the first term, using $f_{\mathrm{dec}}(v) \le Cv^{\epsilon}$ on
$[0,\delta]$ and extending the domain of integration to $[0,\infty)$ (which
only increases the bound, since the integrand is nonnegative),
\[
J_1(\mu) \le C\int_0^{\delta} v^{\epsilon} e^{-\mu v/N_0}\, dv
\le C\int_0^{\infty} v^{\epsilon} e^{-\mu v/N_0}\, dv.
\]
Substituting $u = \mu v/N_0$ (so $v = N_0 u/\mu$, $dv = N_0\,du/\mu$) turns
the right-hand integral into a standard Gamma-function integral:
\begin{align*}
C\int_0^{\infty} \left(\frac{N_0 u}{\mu}\right)^{\epsilon} e^{-u}\,
\frac{N_0}{\mu}\, du
&= C\left(\frac{N_0}{\mu}\right)^{\epsilon+1}\int_0^{\infty} u^{\epsilon} e^{-u}\, du \\
&= C\,\Gamma(\epsilon+1)\left(\frac{N_0}{\mu}\right)^{\epsilon+1}.
\end{align*}
Combining the two pieces,
\[
\tilde{I}_{\mathrm{dec}}(\mu) \le C\,\Gamma(\epsilon+1)\,N_0^{\epsilon+1}\,
\mu^{-(\epsilon+1)} + M(N_0-\delta)\,e^{-\mu\delta/N_0},
\]
and since the second term is exponentially smaller than the first as
$\mu \to \infty$, we obtain:
\[
\tilde{I}_{\mathrm{dec}}(\mu) = O\!\left(\mu^{-(\epsilon+1)}\right), \qquad \mu \to \infty.
\]
As with idle and prefill, the decay is polynomial.

\section*{Acknowledgments}

This work was authored by NLR, operated by Alliance for Sustainable Energy, LLC, for the U.S. Department of Energy (DOE) under Contract No. DE-AC36-08GO28308. Funding provided by DOE Office of Electricity, Advanced Grid Modeling Program. The views expressed in the article do not necessarily represent the views of the DOE or the U.S. Government. The U.S. Government retains and the publisher, by accepting the article for publication, acknowledges that the U.S. Government retains a nonexclusive, paid-up, irrevocable, worldwide license to publish or reproduce the published form of this work, or allow others to do so, for U.S. Government purposes

%
\bibliographystyle{ieeetr}
\bibliography{main_refs_v3}
\end{document}